\documentclass[11pt,a4paper,reqno]{amsart}
\usepackage[a4paper,left=3cm,right=3cm,bottom=4cm,top = 3.5cm]{geometry}

\usepackage[pdftex,bookmarks=true,pdfborder={0 0 0}]{hyperref}

\usepackage[utf8]{inputenc}
\usepackage{graphicx}
\usepackage{amsmath}
\usepackage{amssymb}
\usepackage{mathtools}
\usepackage{enumitem}
\usepackage{color}
\usepackage{microtype}

\usepackage{amsthm}
\newtheorem{theorem}{Theorem}[section]
\newtheorem{lemma}[theorem]{Lemma}
\newtheorem{corollary}[theorem]{Corollary}

\newtheorem{proposition}[theorem]{Proposition}

\theoremstyle{definition}

\newtheorem{problem}[theorem]{Problem}
\newtheorem{remark}[theorem]{Remark}

\newcommand{\beq}{\begin{equation}}
\newcommand{\eeq}{\end{equation}}

\newcommand{\dr}{\mathrm{d}}

\newcommand{\ddop}[1]{\frac{\dr}{\dr #1}}

\newcommand{\abs}[1]{\left\vert#1\right\vert}

\newcommand{\R}{\mathbb{R}}

\newcommand{\N}{\mathbb{N}}

\newcommand{\wto}{\rightharpoonup}

\newcommand{\Mr}{\mathcal{M}_r}

\newcommand{\Lc}{\mathcal{L}}
\renewcommand{\Mc}{\mathcal{M}}

\newcommand{\Vc}{\mathcal{V}}

\newcommand{\mystar}{*}

\providecommand{\abs}[1]{\lvert#1\rvert}

\providecommand{\norm}[1]{\lVert#1\rVert}

\DeclareMathOperator{\rank}{rank}

\DeclareMathOperator{\dist}{dist}

\DeclareMathOperator{\vct}{vec}
\DeclareMathOperator{\spn}{span}
\DeclareMathOperator*{\argmin}{arg\,min}

\numberwithin{equation}{section}

\title{Existence of dynamical low-rank approximations to parabolic problems}

\author[M.~Bachmayr, H.~Eisenmann, E.~Kieri and A.~Uschmajew]{Markus Bachmayr \and Henrik Eisenmann \and Emil Kieri \and Andr{\'e} Uschmajew}
\date{}

\address{Institut f\"ur Mathematik, Johannes Gutenberg-Universit\"at Mainz, 55128 Mainz, Germany}

\address{Max Planck Institute for Mathematics in the Sciences, 04103 Leipzig, Germany}

\address{Institute for Numerical Simulation, University of Bonn, 53115 Bonn, Germany}

\address{Max Planck Institute for Mathematics in the Sciences, 04103 Leipzig, Germany}

\thanks{M.B., E.K., and A.U.{} acknowledge support by the Hausdorff Center of Mathematics, University of Bonn. M.B.\ was supported by the Deutsche Forschungsgemeinschaft (DFG, German Research Foundation) - Projektnummer 211504053 - SFB 1060.}

\subjclass[2010]{Primary 35K15, 35R01; Secondary 15A69, 65L05}

\begin{document}

\begin{abstract}
The existence and uniqueness of weak solutions to dynamical low-rank evolution problems for parabolic partial differential equations in two spatial dimensions is shown, covering also non-diagonal diffusion in the elliptic part. The proof is based on a variational time-stepping scheme on the low-rank manifold. Moreover, this scheme is shown to be closely related to practical methods for computing such low-rank evolutions.
\end{abstract}

\maketitle

\section{Introduction}
\label{sec:intro}

Finding hidden structure in the solutions of partial differential equations has always been a key goal in the study of such equations, whether it is for the sake of modeling or for efficient numerical approximation. In fact, exploiting  structures such as low-dimensional parametrizations can be crucial for the numerical treatment of equations on high-dimensional domains to avoid the curse of dimensionality.

It has been observed that under certain conditions on the domain and the data, the solutions of elliptic and parabolic partial differential equations with a dominating ``Laplacian part'' exhibit low-rank approximability, that is, they can be approximated in certain low-rank tensor formats~\cite{Grasedyck04,Schneider14,Dahmen16,Bachmayr17}. If this is the case, then instead of working on full discretization grids, one can impose the low-rank constraint in the design of the solution method in order to take advantage of low-parametric representations. This usually results in a nonlinear approximation algorithm.

A typical approach is to discretize the partial differential equation on  possibly huge, but finite grids, and then use numerical linear algebra techniques for solving the resulting linear systems in low-rank formats; see, e.g,~\cite{Bachmayr16,Hackbusch19} for an overview and further references. How the obtained solutions behave with refinement of discretization depends strongly on the details of the considered methods. This interaction of low-rank approximations and discretizations has been considered for methods that adjust solution ranks adaptively in each step \cite{Bachmayr15,Bachmayr16b}. For methods based on a fixed low-rank constraint, this question is  more difficult due to the nonlinearity of the resulting constrained problems and has found only limited attention in the literature. Since methods operating on fixed-rank manifolds are important algorithmic building blocks, understanding their robustness under discretization refinement is of high practical interest. A first important requirement is to study the well-posedness of the underlying low-rank problems on infinite-dimensional function spaces. While it is not so difficult to make an appropriate variational formulation for elliptic problems subject to low-rank constraints that ensure existence of solutions ~\cite[Sec.~4]{Bachmayr16}, the parabolic case poses substantial difficulties. In this paper we propose such a formulation for parabolic evolution equations on low-rank manifolds in Hilbert space and prove existence of solutions via a time-stepping scheme.

Dynamical low-rank approximation is a general technique for approximating time-dependent problems under low-rank constraints by projecting the vector field onto the tangent space of the low-rank manifold.
For general initial value problems $\dot Y = F(t, Y)$, $Y(0)=Y_0$ for matrices $Y(t)$, the dynamical low-rank approximation on the manifold $\mathcal{M}_r$ of rank-$r$ matrices as considered in~\cite{Koch07} is given by 
\begin{equation}\label{eq:dynamicallowrank}
   \dot Y(t) = P_{Y(t)} F\bigl(t,Y(t)\bigr),
\end{equation}
where $P_{Y(t)}$ is the orthogonal projector onto the tangent space $T_{Y(t)}\mathcal{M}_r$. 
Note that \eqref{eq:dynamicallowrank} is equivalent to the variational problem
\[
\langle \dot Y(t) - F(t,Y(t)), X\rangle = 0 \quad \text{for all $X \in T_{Y(t)}\mathcal{M}_r$.}
\]
This approach is also known as the Dirac-Frenkel variational principle~\cite{Dirac30,Lubich08}. It has been adapted to several different classes of evolution problems in scientific computing, see, e.g., \cite{Sapsis09,Haegeman11,Musharbash15,Musharbash18,Mena18,Einkemmer18} as well as \cite{Uschmajew20} for an overview, and the monograph \cite{Lubich08} on applications in quantum dynamics. In~\cite{Falco19} the Dirac-Frenkel principle has been rigorously extended to an infinite-dimensional framework of low-rank manifolds in tensor product Banach spaces, with focus on evolution equations in strong formulation.

In this work, we develop a weak formulation of the Dirac-Frenkel principle for low-rank approximation of parabolic problems and prove the existence and uniqueness of solutions in a function space setting.
A model problem representative of the more general class covered by our results is the two-dimensional parabolic equation on the product domain $\Omega = (0,1)^2$,
\beq
\label{eq:parabolic0}
  \begin{aligned}
    u_t(x,t) - \nabla \cdot \alpha(t) \nabla u(x,t) &= f(x,t) \quad & &\text{for $(x,t) \in \Omega \times(0,T)$,} \\
    u(x,t) &= 0 \quad & &\text{for $(x,t)\in \partial \Omega \times (0,T)$,}\\
    u(x,0) &= u_0(x) \quad &&\text{for $x \in \Omega$},
  \end{aligned}
\eeq
for which we consider low-rank approximations separating the two spatial coordinates. Here we assume that the matrix $\alpha(t) = (\alpha_{ij}(t))_{i,j=1,2}$ is symmetric for every $t$, uniformly bounded, and uniformly positive definite. The problem~\eqref{eq:parabolic0} is typically formulated in weak form as follows: given $f \in L_2(0,T;H^{-1}(\Omega))$ and $u_0 \in L_2(\Omega)$, find
\[
 u \in W^1_2(0,T;H^1_0(\Omega),L_2(\Omega)) = \{ u \in L_2(0,T; H^1_0(\Omega)) \colon \, u' \in L_2(0,T; H^{-1}(\Omega)) \}
\]
such that for almost all $t \in (0,T)$,
\beq
  \label{eq:parabolic}
  \begin{aligned}
    \langle u'(t), v\rangle + a(u(t),v;t) &= \langle f(t), v \rangle \quad \text{for all $v \in H^1_0(\Omega)$}, \\
    u(0) &= u_0.
  \end{aligned}
  \eeq
Here, by $\langle\cdot,\cdot\rangle : H^{-1}(\Omega) \times H^1_0(\Omega) \to \R$ we denote the dual pairing of $H^1_0(\Omega)$ and $H^{-1}(\Omega)$, and the symmetric, bounded and coercive bilinear form $a : H_0^1(\Omega) \times H_0^1(\Omega) \times [0,T] \to \R$ is defined as
\begin{multline}\label{eq:definition of a}
 a(u,v;t) = \alpha_{11}(t) \int_\Omega \partial_1 u(x,t)\, \partial_1 v(x,t)\,\mathrm{d}x + \alpha_{22}(t) \int_\Omega \partial_2 u(x,t)\, \partial_2 v(x,t)\,\mathrm{d}x 
 \\
   + \alpha_{12}(t) \int_\Omega \partial_1 u(x,t)\, \partial_2 v(x,t)\,\mathrm{d}x + \alpha_{21}(t) \int_\Omega \partial_2 u(x,t)\, \partial_1 v(x,t)\,\mathrm{d}x.
\end{multline}
By classical theory, the problem~\eqref{eq:parabolic} admits a unique solution; see, e.g.,~\cite[Thm.~23.A]{Zeidler90}.

Since $\Omega = (0,1)^2$, we have $L_2(\Omega) = L_2(0,1) \otimes L_2(0,1)$ in the sense of tensor products of Hilbert spaces, and $H_0^1(\Omega) = H_0^1(0,1) \otimes L_2(0,1) \cap L_2(0,1) \otimes H_0^1(0,1)$ with norm
\[
  \norm{v}_{H^1_0(\Omega)}^2 = \norm{v}^2_{H^1_0(0,1)\otimes L_2(0,1)} + \norm{v}^2_{L_2(0,1)\otimes H_0^1(0,1)}.
\] 
Every function $u  \in L_2(\Omega)$ can be written as 
\begin{equation}\label{eq: low-rank representation}
u(x) = u(x_1,x_2) = \sum_{k} u^1_k(x_1)\, u^2_k(x_2) \quad \text{a.e.,}
\end{equation}
with $u^1_k, u^2_k \in L_2(0,1)$ for all $k$. By $\rank(u)$ we denote the smallest number of non-zero terms that is needed for such a representation to exist; in general, one may have $\rank(u) = \infty$.

As low-rank representations are convenient for several reasons, one can ask whether the parabolic equation~\eqref{eq:parabolic0} admits approximate solutions of low-rank. In dynamical low-rank approximation one assumes this to be the case, and attempts to directly evolve the solution on the set
\begin{equation}\label{eq: low-rank manifold}
  \Mr = \{ u \in L_2(\Omega) \colon \rank(u) = r \}
\end{equation}
for a certain finite value of $r$. Similarly to finite-dimensional fixed-rank matrices, the set $\Mr$ can be shown to be a differentiable manifold in various ways. For instance, it is an immersed Banach manifold as shown in~\cite{Falco19}. In appendix~\ref{app: local manifold structure} we also describe $\Mr$ as a locally submersed and embedded submanifold of $L_2(\Omega)$. In general, when considering manifolds in infinite-dimensional spaces several subtleties may occur; however, we can largely circumvent these in this work. In principle, our framework presented below only requires a closed linear tangent space $T_u \Mr$ at every $u \in \Mr$, such that $T_u \Mr$ contains derivatives of admissible curves through $u$, and certain continuity properties of the corresponding tangent space projectors with respect to~$u$. For the set $\Mr$, the tangent spaces are described in section~\ref{sec: features of model problem}; see~\eqref{eq: tangent space}.

Given the tangent spaces $T_{u} \Mr$, the dynamics on $\Mr$ are then determined by the following problem: find $u \in W^1_2(0,T;H^1_0(\Omega),L_2(\Omega))$ such that
\[
u(t) \in \Mr \quad \text{for all $t \in [0,T]$},
\]
and such that for almost all $t \in (0,T)$,
\beq
  \label{eq:dlr}
  \begin{aligned}
    \langle u'(t), v\rangle + a(u(t),v;t) &= \langle f(t), v \rangle \quad \text{for all $v \in T_{u(t)} \Mr \cap H^1_0(\Omega)$}, \\
    u(0) &= u_0 \in \Mr \,\cap\, H^1_0(\Omega).
  \end{aligned}
\eeq
Thus, in contrast to~\eqref{eq:parabolic}, in~\eqref{eq:dlr} we seek a curve $t \mapsto u(t)$ on $\Mr$ which for almost every $t \in (0,T)$ satisfies the weak parabolic formulation~\eqref{eq:parabolic} on the tangent space only.

Our goal in this paper is to provide an abstract framework for dealing with problems of the type~\eqref{eq:dlr}, and to prove existence of solutions via a time-stepping scheme. We do not require the diffusion matrix $\alpha$ to be diagonal, which means that we allow general anisotropic diffusion. If $\alpha$ is diagonal, that is, $\alpha_{12} = \alpha_{21} = 0$, the problem is substantially easier; in particular, in this case the exact solution of the homogeneous equation with $f = 0$ and $u_0 \in \Mc_r$ satisfies $u(t) = (\exp(t \, \alpha_{11}\, \partial_1^2) \otimes \exp(t \,\alpha_{22} \,\partial_2^2)) u_0 \in \Mc_r$ for all $t$. 
In the case of non-diagonal $\alpha$, the unbounded operator on $L_2(\Omega)$ induced by the bilinear form $a$ no longer maps to the tangent space of the manifold. As a consequence, techniques based on splitting the operator into an unbounded part mapping to the tangent space and an arbitrary bounded part, as previously used, for instance, in \cite{Koch07,Koch07b}, are not applicable in the present setting.

Our existence proof is based on a Rothe-type temporal semidiscretization using mini\-mization problems on $\Mc_r$ (or, more precisely, on its weak closure) in each time step. Off-diagonal parts in the diffusion are treated via bounds on mixed derivatives that are always available for elements in the intersection $\Mc_r \cap H_0^1(\Omega)$, which is a remarkable aspect of the interplay between low-rank structures and regularity in function spaces.
We require slightly more regularity of $u_0$ and $f$ than necessary for standard parabolic problems in linear spaces like~\eqref{eq:parabolic}, but still less than needed for strong solutions. Specifically, applied to the model problem~\eqref{eq:parabolic}, our abstract results give solutions to the dynamical low-rank formulation~\eqref{eq:dlr} under the assumptions $u_0 \in \Mr \cap H_0^1(\Omega)$ and $f \in L_2(0,T; L_2(\Omega))$, as long as the smallest singular values in the low-rank representation of $u(t)$ do not approach zero. 
Compared to previous works, we do not make use of components in low-rank representations, but treat the problem directly on the manifold $\Mr$. This allows for generalization to evolutions on more general manifolds.
We also obtain uniqueness of solutions in the same abstract setting under a mild additional integrability assumption, which 
in particular is satisfied for the model problem~\eqref{eq:dlr}.

Beyond the comparably well-developed analysis of dynamical low-rank approximations in finite-dimensional spaces \cite{Koch07,Arnold14,Kieri16,Feppon18,Ostermann19}, the available results for low-rank evolution problems in function spaces cover mainly Schr\"odinger-type equations \cite{Lubich08}, in particular the closely related higher-dimensional generalization of the multi-configuration time-dependent Hartree method (MCTDH) considered in \cite{Meyer90,Koch07b,Bardos09,Bardos10,Koch11,Falco19}. An important ingredient in many results is the decomposition of the operators into a Laplacian part, which maps points on the low-rank manifold to its tangent space, and a potential term satisfying suitable boundedness properties.
A very similar decomposition with differential operators mapping to the tangent space is also assumed in the recent work \cite{Kazashi20} on parameter-dependent parabolic problems, where the separation of variables is done not between spatial variables as considered here, but rather between the spatial and the parametric variables. An error analysis for such an approach was presented in \cite{Musharbash15}. We also mention a different approach to low-rank approximations of parabolic problems proposed in \cite{Boiveau19}, where separation of spatial and time variables is combined with a space-time variational formulation.

The paper is organized as follows: in section~\ref{sec:abstract-formulation} we give an abstract formulation of the problem for general evolution equations on manifolds under assumptions that reflect the main features of the model problem~\eqref{eq:dlr}. In section~\ref{sec: temporal discretization}, we introduce the time-stepping scheme that is used to approximate solutions. Then we show in section~\ref{sec: existence} that this scheme yields a solution to the continuous problem in the limit, with uniqueness ensured under a minor additional integrability assumption. Section~\ref{sec:numerical methods} is devoted to questions of numerical approximation. We give an outlook on directions for further work in section~\ref{sec:outlook}.

\section{Abstract formulation}
\label{sec:abstract-formulation}

Before we switch to an abstract model for our existence proof, we highlight some particular properties of the model problem~\eqref{eq:dlr} that will motivate the assumptions made in the abstract setting. We believe that the general formulation presented in section~\ref{subsec: abstract setting} will be useful to study parabolic problems on more general low-rank tensor manifolds in tensor product Hilbert spaces of higher order, for instance $L_2((0,1)^d)$, as well. Low-rank tensor formats with suitable properties may include Tucker tensors \cite{DeLathauwer00}, hierarchical Tucker tensors \cite{Hackbusch09}, and tensor trains \cite{Oseledets11}.

\subsection{Some features of the model problem on $\Omega = (0,1)^2$}\label{sec: features of model problem}

Let us first inspect the rank-$r$ manifold $\Mr \subset L_2(\Omega)$ defined in~\eqref{eq: low-rank manifold} in more detail. We note that $\Mr$ is not a closed subset of $L_2(\Omega)$. In fact, its closure $\overline{\Mc_r}$ is the set $\Mc_{\le r}$ of all $u \in L_2(\Omega)$ with $\rank(u) \le r$ and this closure is even weakly sequentially closed; see, e.g., \cite[Lemma~8.6]{Hackbusch19}. In other words,
\[
\Mc_{\le r} = \Mc_{\le r-1} \cup \Mc_r =  \overline{\Mc_r} = \overline{\Mc_r}^w,
\]
where the superscript $w$ indicates the weak sequential closure.
Another important property of $\Mr$ is that it is a cone, that is, $u \in \Mr$ implies $s u \in \Mr$ for all $s > 0$. 

\subsubsection{Tangent spaces}
For convenience let us use the notation $u^1 \otimes u^2$ for the tensor product of two $L_2(0,1)$ functions, that is, $(u^1 \otimes u^2)(x_1,x_2) = u^1(x_1) \,u^2(x_2)$ a.e. Every $u \in \Mr$ admits infinitely many representations of the form~\eqref{eq: low-rank representation}, among which the \emph{singular value decomposition} (SVD)
\begin{equation}\label{eq: SVD}
 u = \sum_{k} \sigma_k^{} u^1_k \otimes u^2_k
\end{equation}
is of central importance for the geometric description of low-rank manifolds. In~\eqref{eq: SVD}, $(u^1_1)$ and $(u^2_k)$ are both $L_2(0,1)$-orthonormal systems, and $(\sigma_k) = (\sigma_k(u))$ is the non-increasing, positive sequence of \emph{singular values}. The existence of such a decomposition is well known in any tensor product of Hilbert spaces~\cite[Thm.~4.137]{Hackbusch19}. The decomposition is rank-revealing in the sense that if $\rank(u) = r$, then $\sigma_1 \ge \dots \ge \sigma_r > 0$, and $\sigma_k = 0$ for $k > r$.

With the SVD~\eqref{eq: SVD}, the \emph{tangent space} at a point $u \in \Mr$, that is, the space of all tangent vectors $\phi '(0)$ to admissible differentiable curves $\phi(t) \in \Mr$ through $u = \phi(0)$, can be written as
\begin{equation}\label{eq: tangent space}
T_u \Mr = \left\{ v = \sum_{k=1}^r v^1_k \otimes u^2_k + u^1_k \otimes v^2_k\, \colon\: v^1_k, v^2_k \in L_2(0,1) \right\}.
\end{equation}
It can be seen quite directly that the space $T_u \Mr$ defined in this way contains only tangent vectors:  
Given $v$ as in~\eqref{eq: tangent space},
consider the curve
\begin{equation}\label{eq: curves in Mc_r}
 \phi(t) = \sum_{k=1}^r \sigma_k \left(u_k^1 + t \sigma_k^{-1} v_k^1\right) \otimes \left( u_k^2 + t \sigma_k^{-1} v_k^2 \right),
\end{equation}
then $\phi(0) = u$ and $\phi'(0) = v$. By~\eqref{eq: best approximation} below, $\phi(t)$ indeed lies in $\Mr$ for small $\abs{t}$, which proves that $v$ is a tangent vector. For a proof that $T_u \Mr$ indeed contains all tangent vectors, we refer to the fact that $T_u \Mr$ is the null space of a local submersion for $\Mr$ as shown in appendix~\ref{app: local manifold structure}. Without loss of generality, we could add the gauging conditions
\begin{equation}\label{eq: gauge conditions}
 ( v^1_k , u^1_\ell )_{L_2(0,1)} = 0 \quad \text{for all $k,\ell$,}  
 \qquad
 ( v^1_k, v^1_\ell )_{L_2(0,1)} = 0 \quad \text{for $k \neq \ell$}
\end{equation}
to the definition of $T_u \Mc_r$, where $(\cdot,\cdot)_{L_2(0,1)}$ is the inner product in $L_2(0,1)$. With these gauging conditions, the representation of tangent vectors becomes unique (all of the terms $v^1_k \otimes u^2_k$ and $u^1_k \otimes v^2_k$ become mutually orthogonal). It is then not difficult to show that $T_u \Mc_r$ is closed in $L_2(\Omega)$. Moreover, subject to~\eqref{eq: gauge conditions}, one can verify that the map $v \mapsto \sum_{k=1}^r \sigma_k \left(u_k^1 + v_k^1\right) \otimes \left( u_k^2 + v_k^2 \right)$, which is a similar construction as~\eqref{eq: curves in Mc_r}, is an embedding from an open neighborhood of zero in $T_u \Mr$ to an open neighborhood  (in the relative $L_2(\Omega)$ topology) of $u$ in $\Mr$.

We will also use the intersection of $\Mc_r$ with smoothness spaces. As shown below, see~\eqref{eq:bound on factor H1 norm}, if $u \in \Mc_r$ belongs to $H_0^1(\Omega)$, then the factors $u_k^1,u_k^2$ in the SVD~\eqref{eq: SVD} all belong to $H_0^1(0,1)$. Likewise, a similar argument shows that if a corresponding tangent vector $v = \sum_{k=1}^r v^1_k \otimes u^2_k + u^1_k \otimes v^2_k$, obeying the gauging conditions~\eqref{eq: gauge conditions}, belongs to $H_0^1(\Omega)$, then the $v_k^1,v_k^2$ are in $H_0^1(0,1)$ as well. Consequently, in this case the curve~\eqref{eq: curves in Mc_r} yielding the tangent vector $v \in T_u \Mc_r  \cap H_0^1(\Omega)$ satisfies 
\begin{equation}\label{eq: curve in V}
\phi(t) \in \Mc_r \cap H_0^1(\Omega)
\end{equation}
for all $t$. An analogous condition will be assumed in the abstract setting as well.

A famous theorem due to Schmidt~\cite{Schmidt1907} states that truncating the SVD of $u$ yields best approximations of lower rank in the $L_2(\Omega)$-norm. A particular instance of this result is that the smallest nonzero singular value $\sigma_r(u)$ of $u \in \Mr$ equals the $L_2(\Omega)$-distance of $u$ to the relative boundary $\Mc_{\le r-1}$ of $\Mc_r$:
\begin{equation}\label{eq: best approximation}
\sigma_{r}(u) = \dist_{L_2(\Omega)}(u,\Mc_{\le r-1}) = \dist_{L_2(\Omega)}(u, \overline{\Mc_r}^w \setminus \Mc_r).
\end{equation}

The smallest singular value is also related to local curvature bounds for the manifold $\Mr$, specifically to perturbations of tangent spaces. For $u \in \Mr$ we denote by $P_u$ the $L_2$-orthogonal projection on $T_u \Mr$. It is given as 
\begin{equation}\label{eq: tangent space projector}
P_u = P_1 \otimes I + I \otimes P_2 - P_1 \otimes P_2
\end{equation}
where $P_1$ and $P_2$ denote the $L_2(\Omega)$-orthogonal projections onto the spans of $u^1_1,\dots,u^1_r$ and $u^2_1,\dots,u^2_r$, respectively. Then one can show the following: for $u, v \in \Mr$ it holds that
\beq
  \label{eq:curv}
  \|P_u - P_v\|_{L_2(\Omega) \to L_2(\Omega)} \le \frac{2} {\sigma_{r}(u)} \|u-v\|_{L_2(\Omega)} 
\eeq
and
\beq
  \label{eq:curv2}
   \|(I - P_v)(u-v)\|_{L_2(\Omega)} \le \frac{1}{\sigma_{r}(u)}\|u-v\|^2_{L_2(\Omega)}.  
\eeq
This behavior of tangent spaces to low-rank manifolds is well known in finite dimension, even for more general tensor formats~\cite{Arnold14,Lubich13,Wei16a}. In infinite-dimensional Hilbert spaces, a bound like~\eqref{eq:curv} and~\eqref{eq:curv2} was obtained, for instance, for the (more general) Tucker format in~\cite{Conte10}. For convenience, we provide a self-contained proof for~\eqref{eq:curv} and~\eqref{eq:curv2} in the appendix 
(Lemma~\ref{thm:curvature} and Corollary~\ref{cor: curvature in H norm}).

With regard to the estimates~\eqref{eq:curv} and~\eqref{eq:curv2}, we further note that on every weakly compact subset $\Mc_r'$ of $\Mc_r$ the infimum
\[
\sigma_* \coloneqq \inf_{u \in \Mc_r'} \sigma_{r}(u) = \inf_{u \in \Mc_r'} \dist_{L_2(\Omega)}(u,\Mc_{\le r-1}) = \dist_{L_2(\Omega)}(\Mc_r',\overline{\Mc_r}^w\setminus \Mc_r)
\]
is positive and attained by some $u_* \in \Mc_r'$. To see this, note first that for Banach spaces, by the Eberlein-{\v S}mulian theorem, weak compactness is equivalent to weak sequential compactness.
Consider sequences $(u_n) \subset \Mc_r'$ and $(v_n) \subset \Mc_{\le r-1}$ such that
\[
\| u_n - v_n \|_{L_2(\Omega)} \le \sigma_* + 1/n.
\]
Both sequences are bounded, and hence $(u_n,v_n)$ admits a weakly converging subsequence with limit $(u_*,v_*)$. Then $u_* \in \Mc_r'$ and $v_* \in \Mc_{\le r-1}$ since both sets are weakly sequentially closed. Since the norm is weakly sequentially lower semicontinuous, we obtain
\(
\sigma_* \le \| u_* - v_* \|_{L_2(\Omega)} \le \sigma_*,
\)
and thus equality. This shows
\begin{equation}\label{eq: sigma*>0}
\sigma_* = \dist_{L_2(\Omega)}(u_*,\Mc_{\le r-1}) > 0.
\end{equation}

\subsubsection{Elliptic operators and low-rank manifolds}

Let us now discuss the interplay between the elliptic operator and the manifold in the model problem~\eqref{eq:dlr}. Note that from the formulation~\eqref{eq:dlr}, we will only have information on $a(u(t),\cdot;t)$ on the tangent space $T_{u(t)} \Mr$. In order to obtain a priori estimates, we thus need additional properties of the bilinear form $a$ that enables us to control $a(u(t),\cdot;t)$ also on the complement of $T_{u(t)} \Mr$.

In case of the model problem~\eqref{eq:dlr} we can split the bilinear form into two parts $a = a_1 + a_2$ with
\[
a_1 = a_{11} + a_{22}, \quad a_2 = a_{12} + a_{21}.
\]
These two parts are generated by the differential operators
\begin{equation}\label{eq:splitting of differential operator}
A_1(t) = - \alpha_{11}(t) \,\partial_1^2 - \alpha_{22}(t) \,\partial_2^2,  \quad A_2(t) = - \alpha_{12}(t) \,\partial_1 \partial_2 - \alpha_{21}(t)\, \partial_2 \partial_1,
\end{equation}
corresponding to divergence and mixed derivatives at time $t$, respectively. The operator $A_1(t)$ has the remarkable property that it maps sufficiently smooth functions $u \in \Mr$ to the tangent space $T_u \Mr$. Namely, given the SVD representation~\eqref{eq: SVD}, we get
\begin{equation}\label{eq: A1 to tangent space}
(A_1(t) u)(x_1,x_2) = - \sum_{k=1}^r \sigma_k(u)  \left[ \alpha_{11}(t)\, \partial_1^2 u_k^1(x_1) \, u_k^2(x_2) + u_k^1(x_1) \, \alpha_{22}(t)\, \partial^2_2 u_k^2(x_2) \right],
\end{equation}
which is in $T_u \Mr$ by~\eqref{eq: tangent space} if the second derivatives $\partial_1^2 u_k^1$ and $\partial_2^2 u_k^2$ are in $L_2(0,1)$.

In order to translate this property to the generated bilinear forms $a_1(\cdot,\cdot;t)$, we observe that if $u \in \Mr \cap H^1_0(\Omega)$, then actually $u \in H^1_{\text{mix}}(\Omega) = H^1_0(0,1) \otimes H^1_0(0,1)$. That is, a finite-rank function $u \in H^1_0 (\Omega)$ automatically possesses mixed derivatives of order one, and all singular vectors $u_k^i$ in the SVD~\eqref{eq: SVD} are themselves in $H^1_0(0,1)$. 
To see this, let $u$ have the SVD~\eqref{eq: SVD}, then, by orthogonality
\[
 u_k^1(x_1) = \frac{1}{\sigma_k} \int_0^1 u(x_1,x_2)\, u^2_k(x_2) \ \mathrm{d}x_2,
\]
which gives
\begin{equation}\label{eq:bound on factor H1 norm}
\| \partial_1 u_k^1\|_{L_2(0,1)} \le \frac{1}{\sigma_k(u)} \| u \|_{H^1_0(\Omega)}.
\end{equation}
Likewise, $\| \partial_2 u_k^2\|_{L_2(0,1)}$ admits precisely the same bound. Note that these bounds could be refined, since, e.g., in~\eqref{eq:bound on factor H1 norm} only the derivative of $u$ with respect to $x_1$ is needed.

Based on the regularity of the singular vectors one can show that if $u \in \Mc_r \cap H_0^1(\Omega)$, the tangent space projection $P_u$ given in~\eqref{eq: tangent space projector} can be bounded in $H_0^1$-norm as a map from $H_0^1(\Omega)$ to $T_u \Mc_r \cap H_0^1(\Omega)$ as follows: 
\begin{equation}\label{eq: H1 projection bound}
\| P_u v \|_{H_0^1(\Omega)} \le \left(1 + \frac{r}{\sigma_r(u)^2} \| u \|_{H_0^1(\Omega)}^2 \right)^{1/2} \| v \|_{H_0^1(\Omega)},
\end{equation}
see Proposition~\ref{cor: bound on projection in V norm} in the appendix.

As a consequence, requiring only $u \in \mathcal{M}_r \cap H^1_0(\Omega)$, we can generalize the feature that the operator $A_1(t)$ maps to the tangent space to the following property of the induced bilinear form $a_1$: for every $t$,
\begin{equation}\label{eq:modelA4a}
a_1(u,v;t) = a_1(u,P_u v;t) \quad \text{for all $u \in \Mc_r \cap H_0^1(\Omega)$ and $v \in H_0^1(\Omega)$.}
\end{equation}
For a proof, choose a sequence $(u_n) \subseteq \Mc_r \cap H^2(\Omega) \cap H_0^1(\Omega)$ converging to $u$ in $H_0^1(\Omega)$-norm. Then for $v \in H_0^1(\Omega)$, we have
\[
a_1(u_n,v;t) = \langle A_1(t) u_n, v\rangle = \langle A_1(t) u_n, P_{u_n} v\rangle = a_1(u_n, P_{u_n} v;t)
\]
since $A_1(t) u_n \in T_{u_n} \Mc$ by~\eqref{eq: A1 to tangent space}. 
Moreover, 
\[
   a_1(u_n, P_{u_n} v;t) = a_1(u, P_u v) + a_1(u, (P_{u_n} - P_u) v) + a_1(u_n - u, P_{u_n} v).
\]
We have $P_{u_n}v\to P_{u}v$ strongly in $L_2(\Omega)$ by~\eqref{eq:curv}, and~\eqref{eq: H1 projection bound} yields $\limsup_n \|P_{u_n}v\|_{H_0^1}<\infty$. Since $L_2(\Omega)$ is dense in $H^{-1}(\Omega)$ it follows that $P_{u_n}v\to P_{u}v$ weakly in $H_0^1(\Omega)$ by a standard argument; see,~e.g.,~\cite[Prop.~21.23(g)]{Zeidler90}. Consequently, $a_1(u_n, P_{u_n} v;t) \to a_1(u, P_u v)$.
At the same time, 
\(
a_1(u_n,v;t) \to a_1(u,v;t)
\), so we have verified~\eqref{eq:modelA4a}.

For the operator $A_2(t)$ on the other hand, the preceding considerations show that it actually is well defined on $\Mr \cap H^1_0(\Omega)$ in a strong sense: applying $\partial_1 \partial_2$ to~\eqref{eq: SVD} and using the triangle inequality we get from~\eqref{eq:bound on factor H1 norm} that
\[
\|\partial_1 \partial_2 u \|_{L_2(\Omega)} \le \sum_{k=1}^r \frac{1}{\sigma_k(u)} \| u \|_{H^1_0(\Omega)}^2 \le \frac{r}{\sigma_{r}(u)} \| u \|_{H^1_0(\Omega)}^2.
\]
By~\eqref{eq:definition of a}, this implies that for every $t$, the bilinear form $a_2(\cdot,\cdot;t)$ associated to the operator $A_2(t)$ has the following property: for fixed $u \in \Mr \cap H^1_0(\Omega)$, the linear functional $v \mapsto a_2(u,v;t)$ on $H^1_0(\Omega)$ is actually continuous on $L_2(\Omega)$, with $L_2(\Omega)$ dual norm 
\beq
\label{eq: estimate for A4b}
\| A_2(t) u \|_{L_2(\Omega)} \le \frac{2 r \abs{\alpha_{12}(t)}}{\sigma_{r}(u)} \| u \|_{H^1_0(\Omega)}^2.
\eeq
Note that here, the inverse of the smallest singular value of $u$ enters again.

\subsection{Abstract formulation of the problem}\label{subsec: abstract setting}

Motivated by the model problem \eqref{eq:dlr} and its properties discussed in the previous section, we next introduce a more general set of assumptions under which we establish existence and uniqueness of solutions. Here we do not explicitly assume that we are dealing with a low-rank manifold, but only use some of its features. Besides standard assumptions on the elliptic part of the operator, we need it to have a decomposition with the basic features of the one in \eqref{eq:splitting of differential operator}.

\subsubsection{Standard assumptions on parabolic evolution equations}
We consider a Gelfand triplet
\[
V \subseteq H \subseteq V^*
\]
where the real Hilbert space $V$ is compactly embedded in the real Hilbert space $H$. Since the embedding is compact it is also continuous, that is,
\beq
  \label{eq:cont-embed}
  \|u\|_H^2 \lesssim \|u\|_V^2 \quad \text{for all } u \in V.
\eeq
In the case $H = L_2(\Omega)$ and $V = H^1_0(\Omega)$, \eqref{eq:cont-embed} is the Poincar{\'e} inequality.

By $\langle\cdot,\cdot\rangle : V^* \times V \to \R$ we denote the dual pairing of $V^*$ and $V$, and by $(\cdot,\cdot)$ we denote the inner product on $H$. For every $t \in [0,T]$, let $a(\cdot,\cdot;t): V \times V \to \R$ be a bilinear form which is assumed to be symmetric,
\[
  a(u,v;t) = a(v,u;t) \quad \text{for all } u, v \in V \text{ and } t \in [0,T],
\]
uniformly bounded,
\[
  |a(u,v;t)| \le \beta\|u\|_V \|v\|_V \quad \text{for all } u, v \in V \text{ and } t \in [0,T]
\]
for some $\beta > 0$, and uniformly coercive,
\[
  a(u,u;t) \ge \mu \|u\|_V^2 \quad \text{for all } u \in V \text{ and } t \in [0,T]
\]
for some $\mu > 0$. Under these assumptions, $a(\cdot,\cdot;t)$ is an inner product on $V$ defining an equivalent norm. Furthermore, it defines a bounded operator
\begin{equation}\label{eq: associated unbounded operator}
 A(t) : V \to V^*
\end{equation}
such that
\[
a(u,v;t) = \langle A(t)u,v \rangle \quad \text{for all $u,v \in V$.}
\]

We also assume that $a(u,v;t)$ is Lipschitz continuous with respect to $t$. In other words, there exists an $L \ge 0$ such that
\beq
  \label{eq:a-lipschitz}
  |a(u,v; t) - a(u,v; s)| \le L\beta\|u\|_V\|v\|_V |t - s|
\eeq
for all $u, v \in V$ and $s, t \in [0, T]$, which in the model problem corresponds to the Lipschitz continuity of the function $t \mapsto \alpha(t)$. 

\subsubsection{Manifolds and tangent spaces}

Our aim is to deal with evolution equations on a (sub)manifold
\[
 \Mc \subseteq H.
\]
For present purposes, we do not have to be very strict regarding the notion of a manifold. 
What we essentially need is a tangent bundle: we assume that for every $u \in \Mc$ there exists a closed subspace $T_u\Mc \subset H$ given via a bounded $H$-orthogonal projection
\[
 P_u : H \to T_u \Mc,
\]
such that $T_u$ contains all tangent vectors to $\Mc$ at $u$. Here a tangent vector is any $v \in H$ for which there exists a (strongly) differentiable curve $\phi : (-\epsilon, \epsilon) \to H$ (for some $\epsilon > 0$) such that $\phi(t) \in \Mc$ for all $t$ and
\[
 \phi(0) = u, \quad \phi'(0) = v.
\]
For our main existence result, we eventually assume that the map $u \mapsto P_u$ is locally Lipschitz continuous on $\Mc$ as a mapping on $H$. 

It will be tacitly assumed that
\begin{itemize}
 \item[--] $\Mc \cap V$ is not empty,
 \item[--] for every $u \in \Mc \cap V$, the space $T_u \Mc \cap V$ is not empty.
\end{itemize}
Indeed, in the main assumptions below we also require that $\Mc$ is a cone, as is the case for low-rank manifolds. Then the first property implies the second, because in this case $u \in T_u \Mc$ for every $u \in \Mc$. 

\subsubsection{Problem formulation and main assumptions}\label{sec:assumptions}

The abstract problem we are considering is now the following.

\begin{problem}\label{problem 1}
Given $f \in L_2(0,T; H)$ and $u_0 \in \Mc \cap V$, find
\[
 u \in W^1_2(0,T;V,H) = \{ u \in L_2(0,T; V) \colon \; u' \in L_2(0,T; H) \}
\]
such that for almost all $t \in [0,T]$,
\beq
  \label{eq:prob1}
  \begin{aligned}
  u(t)  &\in \Mc,\\
    \langle u'(t), v\rangle + a(u(t),v;t) &= \langle f(t), v \rangle \quad \text{for all $v \in T_{u(t)} \Mc \cap V$}, \\
    u(0) &= u_0.
  \end{aligned}
\eeq
\end{problem}

We emphasize again that the main challenge of this weak formulation is that according to the Dirac-Frenkel principle, the test functions are from the tangent space only.
For showing that Problem~\ref{problem 1} admits solutions we will require several assumptions. These assumptions are abstractions of corresponding properties of the model problem of a low rank manifold as discussed in section~\ref{sec: features of model problem}, and hence the main results of this paper apply to this setting. The assumptions are the following.

\medskip

\begin{itemize}[leftmargin=2em,itemsep=.5em]
\item[\textbf{A1}] (Cone property) $\Mc$ is a cone, that is, $u \in \Mc$ implies $s u \in \Mc$ for all $s > 0$. 
\item[\textbf{A2}] (Curvature bound) 
For every subset $\Mc'$ of $\Mc$ that is weakly compact in $H$, there exist constants $\kappa = \kappa(\Mc')$ and $\epsilon = \epsilon(\Mc')$ such that
  \[
    \|P_u - P_v \|_{H \to H} \le \kappa \|u - v\|_H 
  \]
and
\[
\|(I-P_u)(u-v)\|_H\leq\kappa\|u-v\|_H^2 
\]
for all $u,v \in \Mc'$ with $\|u-v\|_H\leq\epsilon$.
\pagebreak
\item[\textbf{A3}] (Compatibility of tangent space)
\smallskip
\begin{itemize}
\item[\upshape (a)]
 For $u \in \Mc \cap V$ and $v \in T_u \Mc \cap V$ an admissible curve with $\phi(0) = u$, $\phi'(0) = v$ can be chosen such that
\[                                                                                                                                                                                      
\phi(t) \in \Mc \cap V
\]
for all $\abs{t}$ small enough.
\item[(b)]
  If $u \in \Mc \cap V$ and $v \in V$ then $P_u v \in T_u \Mc \cap V$.
\end{itemize}
\item[\textbf{A4}] (Operator splitting) The associated operator $A(t)$ in~\eqref{eq: associated unbounded operator} admits a splitting
\[
 A(t) = A_1(t) + A_2(t)
\]
such that for all $t \in [0,T]$, all $u \in \Mc \cap V$ and all $v \in V$, the following holds:
\smallskip
\begin{itemize}
\item[\upshape (a)] ``$A_1(t)$ maps to the tangent space'':
\[
\langle A_1(t)u, v \rangle = \langle A_1(t) u, P_u v \rangle.
\]
 \item[\upshape (b)] ``$A_2(t)$ is locally bounded from $\Mc \cap V$ to $H$'': For every subset $\Mc'$ of $\Mc$ that is weakly compact in $H$, there exists $\gamma = \gamma(\Mc') > 0$ such that
\[ 
A_2(t)u \in H \quad \text{and} \quad \| A_2(t) u \|_H \le \gamma \| u \|_V^\eta \quad \text{for all $u \in \Mc'$}
\]
with an $\eta>0$ independent of $\Mc'$.
\end{itemize}
\end{itemize}

\medskip

Recall that for the model problem~\eqref{eq:dlr} on $\Mr$, \textbf{A2} is stated in~\eqref{eq:curv} and~\eqref{eq:curv2}, taking~\eqref{eq: sigma*>0} into account; see Lemma~\ref{thm:curvature} and Corollary~\ref{cor: curvature in H norm} in the appendix for details. Property \textbf{A3}(a) has been discussed in~\eqref{eq: curve in V}, and \textbf{A3}(b) in~\eqref{eq: H1 projection bound}. 
 With the splitting of $A$ according to~\eqref{eq:splitting of differential operator}, in~\eqref{eq:modelA4a} we have shown that \textbf{A4}(a) holds, and \textbf{A4}(b) follows (with $\eta = 2$ independent of $\Mc'$) from~\eqref{eq: estimate for A4b}, again using~\eqref{eq: sigma*>0} and the boundedness of $\alpha$.

\begin{remark}It is well known that every function $u \in W^1_2(0,T;V,H)$ has a continuous representative $u\in C(0,T;H)$. Yet the notion of solution as defined in Problem~\ref{problem 1} in principle does not require that $u(t) \in \Mc$ for all $t \in [0,T)$. Nonetheless, the existence and uniqueness statements for a maximal time interval (Theorems~\ref{thm: second main theorem} and~\ref{thm: unique}) will be derived, as expected, by extending continuous local solutions until they may hit the boundary of $\Mc$, such that $u(t) \in \Mc$ will in fact be ensured for all $t$ up to this point.
\end{remark}
\begin{remark} 
In Problem~\ref{problem 1} we can actually weaken the uniform coercivity assumption to a uniform G\aa rding inequality 
\begin{align*}
\langle A(t) u,u\rangle\leq \mu \|u\|_V^2-\alpha \|u\|_H^2.
\end{align*}
To see this, suppose $v$ is a solution (in the sense of Problem~\ref{problem 1}) of
\begin{align*}
\langle v'(t)+(A(t)+\alpha I) v(t),w\rangle &=\langle e^{-\alpha t}f(t),w\rangle  \quad \text{for all $w \in T_{v(t)} \Mc \cap V$}, \\
    v(0) &= u_0,
\end{align*}
which, given the G\aa rding inequality, has a uniformly coercive operator. Then $u(t) = e^{\alpha t}v(t)$ solves the equation
\begin{align*}
\langle u'(t)+A(t)u(t),w\rangle &=\langle f(t),w\rangle  \quad \text{for all $w \in T_{v(t)} \Mc \cap V$}, \\
    u(0) &= u_0.
\end{align*}
But since $\Mc$ is a cone, we have $T_{u(t)} \Mc \cap V= T_{v(t)} \Mc \cap V$, that is, $u$ is indeed a solution of Problem~\ref{problem 1} for the initial operator $A(t)$. 
For convenience we can therefore restrict ourselves to the coercive case.
\end{remark}

\begin{remark}
 Our assumptions of course apply to generalizations of \eqref{eq:dlr} derived from \eqref{eq:parabolic0} on higher-dimensional spatial domains of the form $\Omega = \Omega_1 \times \Omega_2$, where $\Omega_i$ for $i=1,2$ is a domain in $\R^{n_i}$ with $n_i \in \N$. In this case, the components in the resulting low-rank representation correspond to groups of $n_1$ and $n_2$ spatial variables, respectively. Note, however, that Assumption \textbf{A4} in general does not apply to parameter-dependent problems where spatial and parametric variables are separated as considered, e.g., in~\cite{Kazashi20} (and in turn, the assumptions made in~\cite{Kazashi20} do not apply in the setting considered here).
\end{remark}

\section{Temporal discretization}\label{sec: temporal discretization}

Given the main assumptions \textbf{A1}--\textbf{A4} stated above, we prove existence of solutions for Problem~\ref{problem 1} by discretizing in time and studying a sequence of approximate solutions with time steps $h \to 0$.
A backward Euler method on $\Mc$ for \eqref{eq:prob1} takes the following form: given $u_i \in \Mc \cap V$ at time step $t_i$, find $u_{i+1} \in \Mc\cap V$ at time step $t_{i+1} > t_i$ such that
\beq
  \label{eq:euler-galerkin}
  \left(\frac{u_{i+1}-u_i}{t_{i+1} - t_i},v\right) + a(u_{i+1},v;t_{i+1}) = \langle f_{i+1},v\rangle \quad\text{for all } v \in T_{u_{i+1}}\Mc\cap V.
\eeq
Here $f_{i+1}$ is the mean value of $f$ on the interval $[t_i,t_{i+1}]$, that is,
\begin{equation}\label{eq:fidef}
f_{i+1} = \frac{1}{t_{i+1}-t_i} \int_{t_i}^{t_{i+1}} f(t) \ \mathrm{d}t.
\end{equation}

As the test space depends on the solution, this equation appears quite difficult to solve. However, when $a(\cdot,\cdot; t_{i+1})$ is symmetric, \eqref{eq:euler-galerkin} is the first order optimality condition of the optimization problem
\beq
  \label{eq:euler-opt}
  u_{i+1} = \argmin_{u\in\overline{\Mc}^w \cap V} F(u) = \frac{1}{2(t_{i+1} - t_i)}\|u-u_i\|_H^2 + \frac{1}{2}a(u,u; t_{i+1}) - \langle f_{i+1}, u\rangle.
\eeq
This is stated in the following lemma.
\begin{lemma}
\label{mincondition}
  Let $u_i\in\Mc\cap V$ and let $u_{i+1}$ be a local minimum of $F$ as in~\eqref{eq:euler-opt} on $\overline{\Mc}^w \cap V$. If $u_{i+1}\in\Mc\cap V$, then \eqref{eq:euler-galerkin} holds.
\end{lemma}

\begin{proof}
  Let $v \in T_{u_{i+1}} \Mc \cap V$. By main assumption \textbf{A3}(a) we can find a differentiable curve $\phi(t)$ defined for $\abs{t}$ small enough such that $\phi(0) = u_{i+1}$, $\phi'(0) = v$ and $\phi(t) \in \Mc \cap V$. Then $t \mapsto F(\phi_v(t))$ has a local minimum at $t=0$ and so the derivative is zero there, which yields~\eqref{eq:euler-galerkin}.
\end{proof}

Next, we consider the existence of minima of~\eqref{eq:euler-opt} on the set $\overline{\Mc}^w \cap V$. This asserts that we can generate approximate solutions $u_1,u_2,\ldots \in \overline{\Mc}^w \cap V$ at a sequence of time steps using~\eqref{eq:euler-opt}, which will serve as the temporal discretization. It will be later ensured that for small enough time steps, we have $u_i \in \Mc \cap V$ if $u_0 \in \Mc \cap V$. Note that in any case the $u_i$ are not uniquely determined from $u_0$, since in general $\overline{\Mc}^w \cap V$ is not convex.

Since the function $F$ in~\eqref{eq:euler-opt} is convex on $V$ and $\overline{\Mc}^w$ is weakly sequentially closed in $H$ by definition, the existence of solutions to~\eqref{eq:euler-opt} is more or less standard.

\begin{lemma}
  \label{thm:existence-opt}
  The optimization problem \eqref{eq:euler-opt} has at least one solution.
\end{lemma}

\begin{proof}
Since $F$ is convex and continuous on $V$ it is also weakly sequentially lower semicontinuous on $V$; see, e.g.,~\cite[Sec.~2.5, Lemma 5]{Zeidler95}. Note that $F$ has bounded sublevel sets on $V$ since the bilinear form $a(\cdot,\cdot;t_{i+1})$ is coercive by assumption. It now follows that $F$ attains a minimum on every weakly sequentially closed subset of $V$ by the standard arguments, since the intersection with a sublevel set remains weakly sequentially compact; see, e.g.~\cite[Prop.~38.12(d)]{Zeidler85}. It hence remains to verify that $\overline{\Mc}^w \cap V$ is weakly sequentially closed in $V$. Consider a sequence $(u_n)\subset \overline{\Mc}^w\cap V$ converging weakly (in $V$) to $u \in V$. Obviously, since $H^* \subseteq V^*$, weak convergence in $V$ implies weak convergence in $H$, and since $\overline{\Mc}^w$ is weakly sequentially closed in $H$, we get $u \in \overline{\Mc}^w \cap V$. This shows that this set is weakly sequentially closed in~$V$.
\end{proof}

\newcommand{\uah}{\hat u_{h}}
\newcommand{\uch}{\hat v_{h}}

\section{Existence and uniqueness of solutions}\label{sec: existence}

In the previous section we defined a time-stepping scheme through a sequence of optimization problems. Starting from $u_0 \in \Mc \cap V$ and setting
\[
h = T/N, \quad t_i = ih,
\]
this generates approximate solutions $u_1,\dots,u_N \in \overline{\Mc}^w \cap V$ at time points $t_i$. In this section we will study the properties of these solutions, and use them to prove existence of solutions to Problem~\ref{problem 1}. Specifically, we construct a function $\uah \colon [0,T] \to V$ by piecewise affine linear interpolation of $u_i$, and another function $\uch \colon [0,T] \to V$ by piecewise constant interpolation of $u_i$ such that $\uch(0) = u_0$ and $\uch(t) = u_i$ on $t \in (t_{i-1},t_i]$. We then verify that the weak limit of these sequences for $h\to 0$ provides a solution of Problem~\ref{problem 1}.

Before turning to their existence, we establish an independent uniqueness result for solutions to Problem~\ref{problem 1}. More precisely, we obtain uniqueness among solutions satisfying some minimal integrability in time, with the order determined by the exponent in assumption~{\upshape\textbf{A4}(b).

\begin{theorem}\label{thm: unique}
Let the assumptions stated in Section~\ref{sec:assumptions} hold and let $u \in W_2^1(0,T^*;V,H)$ be a solution of Problem~\ref{problem 1} on a time interval $[0,T^*]$. Assume that the continuous representative $u\in C(0,T^*;H)$ satisfies $u(t) \in \Mc$ for all $t\in [0,T^*)$. Moreover, assume that $u \in L_\eta(0,T^*;V)$ with $\eta>0$ as in~{\upshape \textbf{A4}(b)}. Then $u$ is the only solution of Problem~\ref{problem 1} in the space $W_2^1(0,T^*;V,H) \cap L_\eta(0,T^*;V)$.
\end{theorem}

\begin{proof}
Let $v$ be another solution of Problem~\ref{problem 1} in the space $W_2^1(0,T^*;V,H) \cap L_\eta(0,T^*;V)$. For both $u$ and $v$ we take the continuous representative in $C(0,T^*;H)$. 
Then there exists $0<t^\mystar\leq T^\mystar$ such that for all $t\in[0,t^\mystar]$ both $u(t)$ and $v(t)$ lie in a weakly compact set $\Mc'\subset \Mc$ and  satisfy $\|u(t)-v(t)\|_H\leq \epsilon(\Mc')$ as in assumption~\textbf{A2}. For almost all $t\in[0,t^\mystar]$,
\begin{align*}
\frac{1}{2}\frac{\dr}{\dr t}\|u(t)-v(t)\|^2_H &\le \langle u'(t)-v'(t) +A(t)(u(t)-v(t)),u(t)-v(t)\rangle\\
&= \begin{multlined}[t] \langle u'(t)+A(t)u(t)-f(t),u(t)-v(t)\rangle - \\
{}-{} \langle v'(t)+A(t)v(t)-f(t),u(t)-v(t)\rangle.
\end{multlined}\\
&= \begin{multlined}[t]\langle u'(t)+A(t)u(t)-f(t),(I-P_{u(t)})(u(t)-v(t))\rangle - \\
 {}-{} \langle v'(t)+A(t)v(t)-f(t),(I-P_{v(t)})(u(t)-v(t))\rangle
\end{multlined}\\
&\leq \begin{multlined}[t] \left(\|u'(t)\|_H+\gamma \|u(t)\|_V^\eta+\|f(t)\|_H\right)\|(I-P_{u(t)})(u(t)-v(t))\|_H + \\
 {}+{} \left( \|v'(t)\|_H+ \gamma \|v(t)\|_V^\eta+\|f(t)\|_H\right) \|(I-P_{v(t)})(u(t)-v(t))\|_H,
\end{multlined}
\end{align*}
where we have made use of the coercivity of $A(t)$,~\eqref{eq:prob1} and assumption~\textbf{A4}. With~\textbf{A2} this further yields
\begin{multline*}
\frac{1}{2}\frac{\dr}{\dr t}\|u(t)-v(t)\|^2_H \le \\
{}\leq{} \kappa \big(\|u'(t)\|_H+\|v'(t)\|_H+\gamma \|u(t)\|_V^\eta
+\gamma \|v(t)\|_V^\eta+2\|f(t)\|_H\big)\|u(t)-v(t)\|_H^2
\end{multline*}
for almost all $t \in [0, t^\mystar]$.
This upper bound is positive and integrable by assumption. Hence Gronwall's Lemma is applicable and yields $\|u(t)-v(t)\|_H=0$ for almost all $t\in [0,t^\mystar]$; by continuity, $u(t) = v(t)$ for all $t \in [0,t^\mystar]$.

Now let $[0,T')$ be a largest interval of the given form on which $u$ and $v$ agree. If $T' < T^*$, then by continuity $v(T') = u(T')$, and by assumption $u(T') \in \Mc$. The above local uniqueness argument can then be repeated from $T'$ on, which yields a contradiction. Hence we must have $T' = T^*$.
\end{proof}}

Note that $W_2^1(0,T^*;V,H) \subset L_\eta(0,T^*; V)$ when $\eta\leq 2$. Hence in this case, the additional integrability requirement in Theorem \ref{thm: unique} is void and we in fact obtain uniqueness in the full space $W_2^1(0,T^*;V,H)$. The integrability requirement is also satisfied for the local solutions obtained in our following main existence result.

\begin{samepage}
\begin{theorem}\label{thm: main theorem}
Let the assumptions stated in Section~\ref{sec:assumptions} hold.
\begin{itemize}
\item[\upshape (a)]
The functions $\uah$ and $\uch$ converge, up to subsequences, weakly in $L_2(0,T;V)$ and strongly in $L_2(0,T;H)$, to the same function $\hat u \in L_\infty(0,T;V)$ with $\hat u(0) = u_0$, while the weak derivatives $\uah'$ converge weakly to $\hat u'$ in $L_2(0,T;H)$, again up to subsequences. We have ${\hat u}(t) \in \overline{\Mc}^w \cap V$ for almost all $t \in [0,T]$.
\item[\upshape(b)]
Let $
\sigma = \dist_H(u_0,\overline{\Mc}^w \setminus \Mc) > 0$.
There exists a constant $c > 0$ independent of $\sigma$ such that ${\hat u}$ solves~\eqref{eq:prob1} for almost all $t < (\sigma/c)^2$, where we set $\dist_H(u_0, \emptyset) = \infty$, and ${\hat u}(t) \in \Mc$ for all $t < (\sigma/c)^2$.  
\end{itemize}
\end{theorem}
\end{samepage}

Note that ${\hat u} \in L_\infty(0,T;V)$ with ${\hat u}' \in L_2(0,T; H)$ implies ${\hat u} \in W^1_2(0,T;V,H)$ in part~(a). A possible constant $c$ in statement (b) is provided by the right hand side of~\eqref{eq:energy-diff} in the energy estimates below, and thus in particular depends continuously on $\norm{u_0}_V$ and $\norm{f}_{L_2(0,T;H)}$. 
In the proof of the theorem, which will be given in sections~\ref{sec: discrete energy estimates}--\ref{sec: proof part b}, we adapt standard techniques for establishing the existence of limits of time discretizations to the abstract manifold setup developed above.

Combining Theorem~\ref{thm: main theorem} with a continuation argument and invoking Theorem~\ref{thm: unique}, we obtain a unique solution on a maximal time interval.

\begin{theorem}\label{thm: second main theorem}
Let the assumptions stated in Section~\ref{sec:assumptions} hold. There exist $T^* \in (0,T]$ and $u \in W_2^1(0,T^*;V,H) \cap L_\infty(0,T^*;V)$ such that $u$ solves Problem~\ref{problem 1} on the time interval $[0,T^*]$, and its continuous representative $u\in C(0,T^*;H)$ satisfies $u(t) \in \Mc$ for all $t \in [0,T^*)$. Here $T^*$ is maximal for the evolution on $\Mc$ in the sense that if $T^* < T$, then 
\[
  \liminf_{t \to T^*} \;\dist_H(u(t), \overline{\Mc}^w \setminus \Mc ) = 0.
\]
In either case, $u$ is the unique solution of Problem~\ref{problem 1} in $W_2^1(0,T^*;V,H) \cap L_\eta(0,T^*;V)$.
\end{theorem}

Recall that for the model problem~\eqref{eq:dlr}, one has $\eta = 2$ by \eqref{eq: estimate for A4b}, and so this problem has a unique solution $u$ in $W_2^1(0,T^*;H_0^1(\Omega),L_2(\Omega))$, and the $u(t)$ are of full rank $r$ in the time interval $[0,T^*)$.

\begin{proof}[Proof of Theorem~\ref{thm: second main theorem}]
The uniqueness of such a solution is immediate from Theorem~\ref{thm: unique}. Hence we only need to show existence. Theorem \ref{thm: main theorem}(b) provides us with a solution $u$ of Problem~\ref{problem 1} on a time interval $[0,T_1]$ with $0<T_1 \leq T$ such that $u \in L_\infty(0,T_1; V)$ and either $T_1 = T$ or $T_1 \geq \frac12 (\sigma_0 / c)^2$ with $\sigma_0 = \dist_H(u_0, \overline{\Mc}^w \setminus \Mc) $ and $c>0$. In the latter case, we may assume without loss of generality that $u \in C(0,T_1;H)$ and $u(T_1) \in \Mc\cap V$. Let $\sigma_1 = \dist_H(u(T_1), \overline{\Mc}^w \setminus \Mc) $. 
 If $T_1 < T$, applying again Theorem~\ref{thm: main theorem} on $[T_1, T]$ with starting value $u_0 = u(T_1)$, we obtain a continuation of $u$ to an interval $[0,T_2]$ with either $T_2 = T$ or $T_2 \geq T_1 + \frac12 (\sigma_1 / c)^2$. In the latter case, we can again assume $u \in C(0,T_2;H)$ and $u(T_2) \in V$ with corresponding distance $\sigma_2>0$. We thus inductively obtain sequences $T_1, T_2, \ldots$ and positive distances $\sigma_1, \sigma_2,\ldots$ which either terminate with $T_i = T$ for some $i$, in which case we are done. Otherwise, $T_i$ is defined for all $i$ and $T_i \to T^* \leq T$. Clearly, the constructed $u \in C(0,T^*;H)$ solves \eqref{eq:prob1} on $[0,T^*)$. If $\inf_i \sigma_i > 0$, then $T_{i+1} - T_i$ is bounded from below, which contradicts $T_i \leq T^*$. Thus $\liminf_{i\to \infty} \sigma_i = 0$, which implies the assertion.
\end{proof}

\subsection{Discrete energy estimates}\label{sec: discrete energy estimates}

First we prove several a priori estimates of the time-discrete solution and its finite differences with respect to time, which are modifications of standard results for time stepping of parabolic PDEs; see, e.g., \cite{Girault79,Emmrich04,Bartels15}. As can be seen from the proof, the assumed cone property \textbf{A1} of $\Mc$ is crucial. 
\begin{lemma}
  \label{thm:energy-discr}
  The sequence $(u_i)_{i=0}^N \subset \overline{\Mc}^w \cap V$ generated by \eqref{eq:euler-opt} with the time step $h = T/N$ satisfies the estimates
  \begin{gather}
    \label{eq:energy-u}
    \|u_N\|_H^2 + \sum_{i=1}^N\|u_i-u_{i-1}\|_H^2 + \mu h\sum_{i=1}^N\|u_i\|_V^2 \le \|u_0\|_H^2 + C_1 \norm{f}_{L_2(0,T;H)}^2, \\
    \label{eq:energy-diff}
    h\sum_{i=1}^N \left\|\frac{u_i-u_{i-1}}{h}\right\|_H^2 \le C_2\Bigl(\|u_0\|_V^2 + \norm{f}_{L_2(0,T;H)}^2  \Bigr), \\
    \label{eq:boundedness in V}
    \|u_i\|_V^2 \le C_3 \Big(\|u_0\|_V^2 + \norm{f}_{L_2(0,T;H)}^2 \Big),\qquad i = 1,\dots,N,
  \end{gather}
  where $C_1, C_2, C_3>0$ depend on $\beta$, $\mu$, $L$, $T$, and on the constant for the continuity of the embedding $V \subseteq H$ in \eqref{eq:cont-embed}. As a result, $\uah$ and $\uch$ are bounded in $L_\infty(0,T;V)$, uniformly for $h \to 0$.    
\end{lemma}

\begin{proof}
  Since $\overline{\Mc}^w$ is a cone, and $u_{i+1}\in\overline{\Mc}^w\cap V$ minimizes $F$ in~\eqref{eq:euler-opt}, it follows directly that $u_{i+1}$ satisfies the optimality condition \eqref{eq:euler-galerkin} with $v = u_{i+1}$ (even when $u_{i+1} \in \overline{\Mc}^w \setminus \Mc$), that is, we have
    \[
    (u_{i+1} - u_i, u_{i+1}) + ha(u_{i+1}, u_{i+1}; t_{i+1}) = h\langle f_{i+1},u_{i+1}\rangle.
  \]
Using the identity
  \[
    (u_{i+1}-u_i, u_{i+1}) = \frac{1}{2}\big(\|u_{i+1}\|^2 - \|u_i\|^2 + \|u_{i+1}-u_i\|^2\big),
  \]
  this reads
  \[
  \|u_{i+1}\|_H^2 - \|u_i\|_H^2 + \|u_{i+1}-u_i\|_H^2 + 2ha(u_{i+1}, u_{i+1}; t_{i+1}) = 2h\langle f_{i+1},u_{i+1}\rangle.
  \]
The coercivity of $a$ implies
\[
    \|u_{i+1}\|_H^2 - \|u_i\|_H^2 + \|u_{i+1}-u_i\|_H^2 + 2h\mu \|u_{i+1}\|_V^2 \le 2h\|f_{i+1}\|_{V^*}\|u_{i+1}\|_V,
 \]
 which
 leads to
 \begin{equation*}
    \|u_{i+1}\|_H^2 - \|u_i\|_H^2 + \|u_{i+1}-u_i\|_H^2 + h\mu \|u_{i+1}\|_V^2 \le \frac{h}{\mu}\|f_{i+1}\|_{V^*}^2,
        \label{eq:e-est-timestep}
  \end{equation*}
  where we have used the geometric mean inequality in the form $2xy \le \mu^{-1} x^2 + \mu y^2$. By summation over $i$ we obtain
     \[
    \|u_N\|_H^2 + \sum_{i=1}^N\|u_i-u_{i-1}\|_H^2 + h\mu \sum_{i=1}^N\|u_i\|_V^2 \le \|u_0\|_H^2 + \frac{h}{\mu}\sum_{i=1}^N\|f_i\|_{V^*}^2.
   \]
The embedding $V \subseteq H$ is continuous, cf.~\eqref{eq:cont-embed}, which implies that also the embedding $H \cong H^* \subseteq V^*$ is continuous. Thus 
\begin{equation}\label{eq:fsumest}
\begin{aligned}
  h \sum_{i=1}^N\|f_i\|_{V^*}^2 & \leq
    C h \sum_{i=1}^N\|f_i\|_{H}^2 \leq
     C h \sum_{i=1} \frac{1}{h^2} \left( \int_{t_{i-1}}^{t_i} 1 \mathrm{d}t \right)\left(  \int_{t_{i-1}}^{t_i} \norm{f(t)}_H^2\,\mathrm{d} t
 \right)    \\
  & \leq
      C \norm { f }_{L_2(0,T; H)}^2
      \end{aligned}
\end{equation}
with a constant $C>0$ depending only on the one in \eqref{eq:cont-embed}, where we have used the Cauchy-Schwarz inequality and the definition~\eqref{eq:fidef} of $f_i$.
This gives~\eqref{eq:energy-u}.
  
  Next we show~\eqref{eq:energy-diff}. Since $u_{i+1}$ minimizes $F$, 
  \begin{align*}
    2F(u_{i+1}) &= \frac{1}{h}\|u_{i+1}-u_i\|_H^2 + a(u_{i+1},u_{i+1}; t_{i+1}) - 2\langle f_{i+1},u_{i+1}\rangle \\
    &\le a(u_i,u_i; t_{i+1}) - 2\langle f_{i+1},u_i\rangle = 2F(u_i),
  \end{align*}
  which can be rearranged to
  \begin{align}
    h\left\|\frac{u_{i+1}-u_i}{h}\right\|_H^2 &\le a(u_i,u_i; t_{i+1}) - a(u_{i+1},u_{i+1}; t_{i+1}) + 2h\left\langle f_{i+1},\frac{u_{i+1}-u_i}{h}\right\rangle \notag \\
    &\le a(u_i,u_i; t_{i+1}) - a(u_{i+1},u_{i+1}; t_{i+1}) + 2h\|f_{i+1}\|_H^2 + \frac{h}{2}\left\|\frac{u_{i+1}-u_i}{h}\right\|_H^2, \notag
    \end{align}
    using a similar trick as above. This yields
    \begin{align}
    \label{eq:to-telescope}
    h\left\|\frac{u_{i+1}-u_i}{h}\right\|_H^2 &\le 2a(u_i,u_i; t_{i+1}) - 2a(u_{i+1},u_{i+1}; t_{i+1}) + 4h\|f_{i+1}\|_H^2.
  \end{align}
  We sum over $i$ and get
  \begin{multline*}
    h\sum_{i=1}^N \left\|\frac{u_i-u_{i-1}}{h}\right\|_H^2 \le 2a(u_0,u_0; 0) + 2\sum_{i=1}^N \big(a(u_{i-1},u_{i-1}; t_i) - a(u_{i-1},u_{i-1}; t_{i-1})\big) \\
    \quad- 2a(u_N,u_N; T) + 4h\sum_{i=1}^N \|f_i\|_H^2.
  \end{multline*}
  Using the Lipschitz continuity \eqref{eq:a-lipschitz} in $t$ of the bilinear form then allows the estimates
  \begin{align*}
    h\sum_{i=1}^N \left\|\frac{u_i-u_{i-1}}{h}\right\|_H^2 &\le 2\beta\|u_0\|_V^2 + 2\beta Lh \sum_{i=1}^N \|u_{i-1}\|_V^2 + 4h\sum_{i=1}^N \|f_i\|_H^2 \\
    &\le 2\beta(1+Lh)\|u_0\|_V^2 + 2\beta Lh \sum_{i=1}^N \|u_i\|_V^2 + 4h\sum_{i=1}^N \|f_i\|_H^2.
  \end{align*}
  By \eqref{eq:energy-u}, which we already proved,
  \begin{align*}
    h\sum_{i=1}^N \left\|\frac{u_i-u_{i-1}}{h}\right\|_H^2 \le 2\beta(1+Lh)\|u_0\|_V^2 + 4h\sum_{i=1}^N \|f_i\|_H^2 
     + \frac{2\beta L}{\mu}\Big(\|u_0\|_H^2 + \frac{h}{\mu}\sum_{i=1}^N\|f_i\|_{V^*}^2\Big).
  \end{align*}
     This allows us to simplify the above expression, and using \eqref{eq:fsumest} we recover~\eqref{eq:energy-diff}.
  
  Finally, we prove~\eqref{eq:boundedness in V}. Starting from \eqref{eq:to-telescope}, we readily obtain
  \[
    0 \le a(u_{j-1},u_{j-1}; t_j) - a(u_j,u_j; t_j) + 2h\|f_j\|_H^2.
  \]
  We sum over $j = 1,\dots,i$ and rearrange:
  \begin{equation*}
    a(u_i,u_i; t_i) \le a(u_0,u_0; 0) + \sum_{j=1}^i\big(a(u_{j-1},u_{j-1};t_j) - a(u_{j-1},u_{j-1};t_{j-1})\big) + 2h\sum_{j=1}^i \|f_j\|_H^2.
     \end{equation*}
     This implies
     \begin{align*}
   \mu\|u_i\|_V^2 &\le \beta\|u_0\|_V^2 + \beta Lh\sum_{j=1}^i\|u_{j-1}\|_V^2 + 2h\sum_{j=1}^i \|f_j\|_H^2 \\
    &\le \beta(1+Lh)\|u_0\|_V^2 + \beta Lh\sum_{j=1}^N\|u_j\|_V^2 + 2h\sum_{j=1}^N \|f_j\|_H^2
    \end{align*}
  for any $i = 1,\dots,N$. Using \eqref{eq:energy-u} and \eqref{eq:fsumest} yields~\eqref{eq:boundedness in V}.
\end{proof}

\begin{remark}
In standard estimates of the solution on the full linear space, the difference quotient \eqref{eq:energy-diff} is typically bounded in $L_2(0,T; V^*)$ in terms of $\|u_0\|_H$ and $\|f\|_{L_2(0,T;V^*)}$, cf.~\cite{Emmrich04}. One then uses the boundedness of $a(\cdot,\cdot; t)$ and $f$ to get
\[
  \left\langle\frac{u_{i+1}-u_i}{h}, v\right\rangle = -a(u_{i+1},v; t_{i+1}) + \langle f_{i+1},v\rangle \le \beta\|u_{i+1}\|_V\|v\|_V + \|f_{i+1}\|_{V^*}\|v\|_V.
\]
Dividing by $\|v\|_V$ and taking the supremum over $V\backslash\{0\}$ gives
\begin{align*}
  \left\|\frac{u_{i+1}-u_i}{h}\right\|_{V^*} \le \beta\|u_{i+1}\|_V + \|f_{i+1}\|_{V^*}
  \end{align*}
  and
\begin{align*}
  h\sum_{i=1}^N\left\|\frac{u_i-u_{i-1}}{h}\right\|_{V^*} \lesssim \Big(\|u_0\|_H + h\sum_{i=1}^N\|f_i\|_{V^*}\Big).
\end{align*}
However, we can not do this for solutions constrained to $\Mc$. Since the difference quotient is not necessarily in the tangent space, testing only with the tangent space does not give us the supremum and thus not the dual norm. We use a different reasoning in Lemma~\ref{thm:energy-discr}, testing with $v = u_{i+1}$, where $u_{i+1} \in T_{u_{i+1}} \Mc$ by the cone property~\textbf{A1}. We thus obtain a bound in $L_2(0,T; H)$-norm in terms of $\|u_0\|_V$ and $\|f\|_{L_2(0,T;H)}$ instead.
\end{remark}

\subsection{Proof of Theorem~\ref{thm: main theorem}(a)}

We now prove statement (a) of Theorem~\ref{thm: main theorem}. The argument for showing the existence of the limiting function ${\hat u}$ relies on standard compactness arguments based on the energy estimates in Lemma~\ref{thm:energy-discr}. Showing that ${\hat u}(t) \in \overline{\Mc}^w$ for almost all $t$ is then based on the fact that this set is weakly sequentially closed by definition. 

It follows from Lemma~\ref{thm:energy-discr} that $\uah$ and $\uch$ are bounded in $L_2(0,T;V)$, uniformly with respect to $h$. Therefore, refinement in time generates sequences in $L_2(0,T;V)$ which, up to subsequences, converge weakly,
\[
  \uah \wto {\hat u} \quad \text{and} \quad \uch \wto \hat v \quad \text{in } L_2(0,T;V).
\]
In particular, $\uah - \uch$ converges weakly in $L_2(0,T;H)$ to ${\hat u} - \hat v$. Comparing the two sequences in $L_2(0,T;H)$, we get
\begin{align*}
  \int_0^T \|\uah - \uch\|_H^2\,\dr t &= \sum_{i=1}^N\int_{t_{i-1}}^{t_i} \|\uah - \uch\|_H^2\,\dr t \\
  &= \sum_{i=1}^N\int_{t_{i-1}}^{t_i} \left\| \left(\frac{t_i-t}{h}\right) u_{i-1} + \left(\frac{t-t_{i-1}}{h}\right) u_i - u_i \right\|_H^2\,\dr t \\
  &= h\sum_{i=1}^N\int_0^1 \|(s-1)(u_i - u_{i-1})\|_H^2\,\dr s \\
  &= \frac{h}{3} \sum_{i=1}^N\|u_i - u_{i-1}\|_H^2,
\end{align*}
and by Lemma~\ref{thm:energy-discr},
\beq
  \label{eq:Uh-Vh-equal}
  \int_0^T \|\uah - \uch\|_H^2\,\dr t \le \frac{C_2 h^2}{3}\Bigl(\|u_0\|_V^2 + \norm{f}_{L_2(0,T;H)}^2  \Bigr),
\eeq
which tends to zero as $h \to 0$. We conclude ${\hat u} = \hat v$.

Likewise, $\uah'$ is uniformly bounded in $L_2(0,T;H)$ and thus, up to subsequences, $\uah' \wto \hat w$ for some $\hat w \in L_2(0,T;H)$. We next show that $\hat w$ is the weak derivative of ${\hat u}$. For this, we need to verify that
\[
  \int_0^T ( \hat w(t), v)\, \phi(t)\,\dr t + \int_0^T ({\hat u}(t), v)\,\phi'(t)\,\dr t = 0
\]
for arbitrary $v \in V$ and $\phi \in C_0^\infty (0,T)$. 
Adding and subtracting the weak derivative of $\uah$, we get
\begin{multline*}
  \int_0^T ( \hat w(t), v )\,\phi(t)\,\dr t + \int_0^T ({\hat u}(t), v)\,\phi'(t)\,\dr t = \\
  {}={} \int_0^T ( \hat w(t) - \uah'(t), v) \,\phi(t)\,\dr t + \int_0^T ({\hat u}(t) - \uah(t), v)\,\phi'(t)\,\dr t.
\end{multline*}
Since $\uah \wto {\hat u}$ and $\uah' \wto \hat w$ in $L_2(0,T; V)$ and $L_2(0,T; H)$, respectively, and since $v\phi$, $v\phi' \in L_2(0,T; V)$, the right hand side converges to zero. Thus, $\hat w = {\hat u}'$.

The strong convergence in $L_2(0,T; H)$ of $\uah$ follows from the theorem of Lions and Aubin \cite[Prop.~III.1.3]{Showalter97}. It states that when $V$ is compactly embedded in $H$, then the space $W_2^1(0,T; V, H)$ is compactly embedded in $L_2(0,T;H)$. Thereby, the weak convergence of $\uah$ and $\uah'$ that we just have proven implies the strong convergence $\uah \to {\hat u}$ in $L_2(0,T;H)$. This together with \eqref{eq:Uh-Vh-equal} directly proves that also $\uch \to {\hat u}$ in $L_2(0,T;H)$. By \eqref{eq:boundedness in V} and lower semicontinuity of the $L_\infty(0,T;V)$-norm with respect to weak convergence in $L_2(0,T;V)$, we even obtain ${\hat u} \in L_\infty(0,T;V)$.

It remains to show that ${\hat u}(t) \in \overline{\Mc}^w$ for almost all $t\in[0,T]$. Recall that $h = T/N$, and let $t_i^{(N)} = ih$, $i = 0,\dots,N$. To any fixed $t \in [0,T]$ we associate a sequence $(t_{j_N}^{(N)})_{N=N_0}^\infty \subset [0,T]$ such that $t_{j_N}^{(N)} \to t$ as $N \to \infty$. We can take $t_{j_N}^{(N)}$ as the largest possible $t_i^{(N)} \le t$, which implies $0 \le t - t_{j_N}^{(N)} \le h$. 
If we can show that for almost all $t \in [0,T]$ a subsequence of $(\uah(t_{j_N}^{(N)})) \subseteq \overline{\Mc}^w$ converges weakly in $H$ to ${\hat u}(t)$, we then get that ${\hat u}(t) \in \overline{\Mc}^w$ for such~$t$. 
We will even show that there exists a strongly convergent subsequence based on the inequality
  \[
    {\bigl\|\uah(t_{j_N}^{(N)}) - {\hat u}(t)\bigr\|}_H \le {\bigl\|\uah(t_{j_N}^{(N)}) - \uah(t)\bigr\|}_H + {\bigl\|\uah(t) - {\hat u}(t)\bigr\|}_H.
  \]
For the second term on the right hand side, since $\uah \to {\hat u}$ in $L_2(0,T; H)$, and possibly passing to a subsequence, we have $\uah(t) \to {\hat u}(t)$ in $H$ for almost all $t$. In order to show that the first term of the right hand side vanishes in the limit we recall that, by construction, $\uah$ is linear on the interval $[t_{j_N}^{(N)}, t_{j_N+1}^{(N)}]$ that contains the given $t$. Therefore, using \eqref{eq:energy-diff},
  \begin{align*}
    {\bigl \|\uah(t_{j_N}^{(N)}) - \uah(t)\bigr\|}_H &\le {\bigl\|\uah(t_{j_N}^{(N)}) - \uah(t_{j_N+1}^{(N)})\bigr\|}_H
     \leq \sqrt{ C_2 h \Big({\|u_0\|}_V^2 +  \norm{f}_{L_2(0,T; H)}^2\Big)},
  \end{align*}
  which vanishes in the limit.
  This shows $\uah(t_{j_N}^{(N)}) \to \hat u(t)$ strongly in $H$ for almost all $t$.

Finally, we show that ${\hat u}(0) = u_0$. By construction, $\uah(0) = u_0$. Choosing $v \in C^\infty(0,T;V)$ such that $v(T) = 0$ and applying integration by parts gives
\[
  \int_0^T \langle \uah'(t),v(t)\rangle\,\dr t + \int_0^T\langle \uah(t),v'(t)\rangle\,\dr t = -(\uah(0),v(0)) = -(u_0,v(0)).
\]
In the limit $h \to 0$,
\begin{multline*}
  -(u_0,v(0)) = \int_0^T \langle \uah'(t),v(t)\rangle\,\dr t + \int_0^T\langle \uah(t),v'(t)\rangle\,\dr t \to \\
   \to \int_0^T \langle {\hat u}'(t),v(t)\rangle\,\dr t + \int_0^T\langle {\hat u}(t),v'(t)\rangle\,\dr t = -({\hat u}(0),v(0)),
\end{multline*}
and as $(u_0,v(0))$ is independent of $h$, $(u_0 - {\hat u}(0),v(0)) = 0$ for all $v(0) \in V$.

  This concludes the proof of Theorem~\ref{thm: main theorem}(a). \hfill \qedsymbol

\subsection{Proof of Theorem~\ref{thm: main theorem}(b)}\label{sec: proof part b}

Our goal is to show that there exists a constant $c>0$ such that for all  $\theta \in (0,1)$ and
\[
T_\theta = \min \left\{  \left(\frac{\theta \sigma}{c} \right)^2 , T \right\},
\]
the limiting function ${\hat u}(t)$ solves Problem~\ref{problem 1} for almost all $t \in [0, T_\theta]$. Since Problem~\ref{problem 1} is formulated on $\Mc \cap V$, we first in particular need to ensure that ${\hat u}(t) \in \mathcal{M}$ for almost all~$t$. We do this by showing next that the $\uch(t)$ keep a positive distance in $H$-norm to $\overline{\Mc}^w \setminus \Mc$. For fixed $h = T/N$ the estimate~\eqref{eq:energy-diff} in Lemma~\ref{thm:energy-discr} yields for every integer $j \le N$ that
\begin{align*}
\| u_j - u_0 \|_H \le \sum_{i=1}^j \| u_i - u_{i-1} \|_H &\le \sqrt{j} \left( \sum_{i=1}^j \| u_i - u_{i-1} \|_H^2 \right)^{1/2}\\
&= \sqrt{t_j} \left( \sum_{i=1}^j \frac{\| u_i - u_{i-1} \|_H^2}{h} \right)^{1/2}\\
&\le \sqrt{t_j}\, c,
\end{align*}
where $c$ is the right hand side of~\eqref{eq:energy-diff}. Using this $c$ in the definition of $T_\theta$, we have ensured
\[
\| u_j - u_0 \|_H \le \theta \sigma \quad \text{for all $t_j \le T_\theta \leq  \frac{\theta^2 \sigma^2}{c^2}$.}
\]
Hence, by construction, since $\uch$ is the piecewise constant interpolant,
\[
\| \uch(t) - u_0 \|_H \le \theta \sigma \quad \text{for all $t \le T_\theta$.}
\]
Recalling that $\sigma$ is the distance of $u_0 \in \Mc$ to $\overline{\Mc}^w \setminus \Mc$, this shows that 
\begin{equation}\label{eq: positive distance to boundary}
\uch(t) \in \Mc' \coloneqq \{ u \in \Mc \colon \| u - u_0 \|_H \le \theta \sigma \} \quad \text{for all $t \le T_\theta$.}
\end{equation}
Indeed, the set $\Mc'$ coincides with the set $\{ u \in \overline{\Mc}^w \colon \| u - u_0 \|_H \le \theta \sigma \}$, which shows that it is a weakly compact subset of $\Mc$. 
Since, up to subsequences, $\uch(t) \to {\hat u}(t)$ strongly in $H$ for almost all $t$ (by part (a)), we get
${\hat u}(t) \in \Mc' \subset \Mc$ for almost all $t \in [0,T_\theta]$.
Since ${\hat u} \in C(0,T; H) \subset W^1_2(0,T; V,H)$, we obtain 
\[
 {\hat u}(t) \in \Mc\quad \text{for all $t \in [0,T_\theta]$.} 
\]

Note next that~\eqref{eq: positive distance to boundary} holds independently of $h$.
The main assumptions ~\textbf{A2} and ~\textbf{A3} therefore provide us with positive constants $\eta$, $\gamma$, and $\kappa$ such that
\begin{equation*}\label{A2est}
\| A_2(t) \uch(t) \|_H \le \gamma \| \uch(t) \|_V^\eta
\end{equation*}
and
\begin{equation}\label{Pest}
\| P_{{\hat u}(t)} - P_{\uch(t)} \|_{H \to H} \le \kappa \| {\hat u}(t) - \uch(t) \|_H 
\end{equation}
whenever $\norm{{\hat u}(t)-\uch(t)}_H\leq\epsilon(\Mc')$, for all $h$ and almost all $t \in [0,T_\theta]$. These are the crucial estimates in order to show that ${\hat u}(t)$ solves Problem~\ref{problem 1} for all such $t$ in the remainder of this proof.

Using the piecewise constant interpolant
\[
  F_h(t) = f_i, \quad t_{i-1} < t \le t_i,\ i = 1, \dots,N,
\]
the Galerkin-type condition~\eqref{eq:euler-galerkin} can be written as
\begin{equation}\label{eq: euler-galerkin 2}
  \langle \uah'(t),v\rangle + a_h(\uch(t),v;t) = \langle F_h(t),v\rangle \quad \text{for all } v \in T_{\uch(t)}\Mc \cap V,
  \end{equation}
where $a_h(\cdot,\cdot)$ is the piecewise constant in time interpolant of $a(\cdot,\cdot; t)$. By Lemma~\ref{mincondition}, this holds as long as $\uch(t) \in \Mc$, which by our above considerations is ensured for all $t \in [0,T_\theta]$. 
  For these $t$, we define the spaces
  \begin{equation*}
    \Vc(t) = T_{{\hat u}(t)}\Mc \cap V,
    \quad     \Vc_h(t) = T_{\uch(t)}\Mc \cap V.
  \end{equation*}
We need to show that
  \begin{equation}\label{aim}
    \langle \Lc(t;{\hat u}),v \rangle \coloneqq \langle {\hat u}'(t),v\rangle + a({\hat u}(t),v; t) - \langle f(t),v\rangle = 0, \quad v \in \Vc(t),
  \end{equation}
for almost all $t\in[0,T_\theta]$, and also that ${\hat u}(0) = u_0$. 
Consider the related expression
\beq
  \label{eq:appr-variational}
  \langle \Lc_h(t;\uah,\uch),v \rangle \coloneqq \langle \uah'(t),v\rangle + a_h(\uch(t),v; t) - \langle F_h(t),v\rangle
\eeq
for an arbitrary $v\in\Vc(t)$. As $v$ is in the tangent space at ${\hat u}(t)$, and not at $\uch(t)$, this expression in general does not equal zero exactly. With test functions in the correct tangent space, however, we do recover \eqref{eq:euler-galerkin}, that is, we have
\begin{equation}\label{discrgalerkin}
   \langle \Lc_h(t;\uah,\uch),v_h \rangle = 0,\quad  v_h \in \Vc_h(t).
\end{equation}

Our first goal is to show, term by term, that for any $w \in L_2(0,T_\theta; V)$ satisfying $w(t) \in \Vc(t)$ for almost all $t$, we have
\begin{equation}\label{conv1}
  \int_0^{T_\theta} (\Lc_h(t;\uah,\uch),w(t))\,\dr t \to \int_0^{T_\theta} (\Lc(t;{\hat u}),w(t))\,\dr t\,.
\end{equation}
as $h \to 0$. For the first term in the right hand side of \eqref{eq:appr-variational}, we immediately obtain
\[
  \int_0^{T_\theta} \langle \uah'(t),w(t)\rangle\,\dr t \to \int_0^{T_\theta} \langle {\hat u}'(t),w(t)\rangle\,\dr t.
\]
Regarding the second term, the bilinear form $a_h(\cdot,\cdot; t)$ defines an operator $A_h(t):V\to V^*$,
\[
  a_h(\uch(t),w(t);t) = \langle A_h(t)\uch(t),w(t)\rangle,
\]
  and since $a_h(\cdot,\cdot; t)$ is symmetric,
\[
  a_h(\uch(t),w(t); t) = a_h(w(t), \uch(t); t) = \langle A_h(t)w(t), \uch(t)\rangle.
  \]
We then get
\begin{multline}
  \int_0^{T_\theta} a_h(\uch(t),w(t);t)\,\dr t = \int_0^{T_\theta} \langle A_h(t)w(t), \uch(t)\rangle\,\dr t = \\
  {}={} \int_0^{T_\theta} \langle A(t)w(t), \uch(t)\rangle\,\dr t + \int_0^{T_\theta} \langle (A_h(t)-A(t))w(t), \uch(t)\rangle\,\dr t. \label{eq:A-int-Lip}
\end{multline}
We have $Aw, A_h w \in L_2(0,T; V^*)$, and hence
\begin{align*}
  &\int_0^{T_\theta} \langle A(t)w(t), \uch(t)\rangle\,\dr t \to \int_0^{T_\theta} \langle A(t)w(t), {\hat u}(t)\rangle\,\dr t = \int_0^{T_\theta} a({\hat u}(t),w(t); t)\,\dr t 
\end{align*}
as $h \to 0$. The second integral in \eqref{eq:A-int-Lip} vanishes in the limit, since
\begin{multline*}
  \Big|\int_0^{T_\theta} \langle (A_h(t)-A(t))w(t), \uch(t)\rangle\,\dr t\Big| \le \int_0^{T_\theta} |\langle (A_h(t)-A(t))w(t), \uch(t)\rangle|\,\dr t \le \\
  {}\le{} \int_0^{T_\theta} hL\beta \|w(t)\|_V \|\uch(t)\|_V\,\dr t \le hL\beta \|w\|_{L_2(0,T_\theta;V)}\|\uch\|_{L_2(0,T_\theta;V)} \to 0.
\end{multline*}
We thus have shown
\[
  \int_0^{T_\theta} a_h(\uch(t),w(t);t)\,\dr t \to \int_0^{T_\theta} a({\hat u}(t),w(t); t)\,\dr t.
\]
Finally, as $F_h$ is a piecewise constant interpolant of a given function $f \in L_2(0,T_\theta;H)$, we have $F_h \to f$ strongly in $L_2(0,T_\theta;H)$, and altogether we obtain~\eqref{conv1}. 

We next show that for all $w$ as above,
\begin{equation}\label{conv2}
  \int_0^{T_\theta} \langle \Lc_h(t;\uah,\uch),w(t)\rangle \,\dr t \to 0
\end{equation}
as $h \to 0$. By \eqref{discrgalerkin}, for $v \in \Vc(t)$ and $v_h \in \Vc_h(t)$, it holds that
\[
  \langle \Lc_h(t;\uah,\uch),v\rangle
  = \langle \Lc_h(t;\uah,\uch), v - v_h\rangle.
\]
We choose $v_h = P_{\uch(t)}v$. Note that $v = P_{{\hat u}(t)}v$ for $v\in\Vc(t)$ for almost all $t \in [0,T_\theta]$. Thus
\begin{equation*}\label{Lhprod}\begin{aligned}
  \langle \Lc_h(t;\uah,\uch), v\rangle  &= \langle \Lc_h(t;\uah,\uch), (P_{{\hat u}(t)} - P_{\uch(t)})v\rangle  \\
  &=  \langle \uah'(t) + A_h(t)\uch(t) - F_h(t), (P_{{\hat u}(t)} - P_{\uch(t)})v\rangle \\
  &= \langle \uah'(t) + A_h(t)\uch(t) - F_h(t), (I - P_{\uch(t)})(P_{{\hat u}(t)} - P_{\uch(t)})v \rangle,
\end{aligned}\end{equation*}
where the last equality holds due to~\textbf{A3}(b) and~\eqref{eq: euler-galerkin 2}. In light of $\textbf{A4}$ we hence have the estimate
\[
\abs{\langle \Lc_h(t;\uah,\uch), v\rangle } \le \bigl( \norm{\uah'(t)}_H + \gamma \norm{\uch(t)}_V^\eta + \norm{F_h(t)}_H \bigr) \norm{(P_{{\hat u}(t)} - P_{\uch(t)})v}_H   
\]
for almost all $t \in [0,T_\theta]$.
By the curvature bound~\eqref{Pest},
\begin{equation}\label{curvaturebound with D}
    \norm{(P_{{\hat u}(t)} - P_{\uch(t)})v}_H \leq D_{t,h} \norm{v}_H ,\quad D_{t,h}\coloneqq  \kappa \norm{{\hat u}(t)-\uch(t)}_H,
\end{equation}
once $\norm{{\hat u}(t)-\uch(t)}_H\leq\epsilon(\Mc')$. 
Since $\uch$ converges strongly in $L_2(0,T_\theta; H)$, up to passing to a subsequence we have $\norm{\hat u(t) - \uch(t)}_H \to 0$ for almost all $t \in [0,T_\theta]$. Hence for almost all $t$, \eqref{curvaturebound with D} applies for sufficiently small $h$.
We can therefore conclude that
\[
 \langle \Lc_h(t;\uah,\uch), v\rangle  \to 0 \quad \text{for almost all $t \in [0,T_\theta]$}
\]
when $v \in \Vc(t)$. Specifically taking $v = w(t)$ in the above considerations shows 
\[
 | \langle \Lc(t;\uah,\uch),w(t) \rangle | \leq D_{t,h}( \norm{\uah'(t)}_H + \gamma \norm{\uch(t)}_V^\eta + \norm{F_h(t)}_H) \norm{w(t)}_H,
\]
where $D_{t,h}$ and $\norm{\uch(t)}_V$ are bounded uniformly in $h$ for almost all $t\in [0,T_\theta]$. Hence the right hand side provides an integrable upper bound, and by the dominated convergence theorem we arrive at~\eqref{conv2}.

Combined with \eqref{conv1}, we conclude
\begin{equation}\label{intzero}
     \int_0^{T_\theta} \langle \Lc(t;{\hat u}),w(t)\rangle \,\dr t = 0  \quad \text{for all $w \in L_2(0,T_\theta; V)$, $w(t) \in \Vc(t)$.}
\end{equation}
This shows \eqref{aim} as desired, since in the opposite case there would be a subset $S \subseteq [0,T_\theta]$ of positive measure such that for all $t \in S$ we have $\langle \Lc(t;{\hat u}),v\rangle \neq 0$ for some $v\in \Vc(t)$. By appropriately scaling these $v$, we can then choose $w(t) \in \Vc(t)$ such that $\norm{w}_{L_\infty(0,T_\theta;V)} < \infty$ and $\langle \Lc(t;{\hat u}),w(t) \rangle >0$ (since $\Vc(t)$ is a linear space). Hence the left hand side of \eqref{intzero} would be positive.

This concludes the proof of Theorem~\ref{thm: main theorem}(b). \hfill \qedsymbol

\section{Numerical Methods for Low-Rank Matrix Manifolds}
\label{sec:numerical methods}

In this section, we comment on how the basic variational time stepping scheme \eqref{eq:euler-opt}, which we have used to prove the existence of a solution to Problem~\ref{problem 1}, is connected to numerical methods for actually computing the low-rank evolution.
A first strategy for solving the general dynamical low-rank problem \eqref{eq:dynamicallowrank}, as used in \cite{Koch07}, is to extract from \eqref{eq:dynamicallowrank} equations for the components $U$, $S$, $V$ in a factorization $Y(t) = U(t)\,S(t)\,V(t)^T$, where $U(t), V(t)$ have orthonormal columns, and $S(t) \in \R^{r\times r}$ is invertible, but not necessarily diagonal.
These equations can then be solved by standard time stepping schemes. Note that in this section we reserve the notation $U$ and $V$ for matrix factors in order to adopt to standard notation in linear algebra and to avoid too many indices.

An alternative scheme was proposed in \cite{Lubich14}. 
For notational convenience, the above matrix factorization $Y = USV^T$ can be written in vectorized form as
\[
y\coloneqq\vct(Y) = \vct(USV^T) = (V \otimes U)s, \quad s = \vct(S).
\]
With $P_U \coloneqq U U^T$ and $P_V \coloneqq V V^T$, \eqref{eq:dynamicallowrank} is then rewritten as
\begin{equation}\label{eq: projector dynamical low rank}
    y'(t) = (I \otimes P_{U(t)}) F\bigl(t,y(t)\bigr) - (P_{V(t)} \otimes P_{U(t)}) F\bigl(t,y(t)\bigr) + (P_{V(t)} \otimes I) F\bigl(t,y(t)\bigr),
\end{equation}
based on the formula~\eqref{eq: tangent space projector} for the tangent space projector. A time stepping scheme is then obtained by applying an operator splitting, that is, by integrating the three terms on the right hand side of~\eqref{eq: projector dynamical low rank} in time in the given order. 
As shown in \cite{Lubich14}, the resulting method has very interesting characteristics; for instance, the splitting is exact if $F\bigl(t,Y(t)\bigr)$ is in the tangent space at $Y(t)$ on the considered time interval.

Let $( \varphi_n )_{n \in \mathcal{I}}$ with $\mathcal{I}\subseteq \N$ be an orthonormal system in $L_2(0,1)$ with all $\varphi_n$ sufficiently regular, and let 
\[
  \tilde A(t) \coloneqq \Bigl( \bigl\langle A(t) (\varphi_{j_1}\otimes\varphi_{j_2}) , (\varphi_{i_1}\otimes \varphi_{i_2}) \bigr\rangle \Bigr)_{i,j\in \mathcal{I}^2},\quad 
  \tilde f(t) \coloneqq \Bigl( \bigl\langle f(t),  (\varphi_{i_1}\otimes \varphi_{i_2})  \bigr\rangle \Bigr)_{i\in \mathcal{I}^2},
\]
so that the initial value problem
\[
   \tilde u'(t) + \tilde A(t) \tilde u(t) = \tilde f(t)
\]
for $\tilde u(t) \in \ell_2(\mathcal{I})$ is a Galerkin semidiscretization or, if $( \varphi_n)$ is an orthonormal basis, the basis representation of~\eqref{eq:parabolic}.
The splitting scheme from \cite{Lubich14} for this problem, for a time step of length $h$ with low-rank initial data $u_0 = (V_0 \otimes U_0)s_0$, formally reads as follows:

\medskip

\begin{subequations}\label{eq:splitting}
\begin{itemize}[leftmargin=1em,itemsep=1em]
\item[--] Determine $U_1 = U(h)$, $s^+_1 = s(h)$ as solution of 
  \begin{equation}\label{eq:splitting1}
    \ddop{t}\bigl(V_0 \otimes U(t)\bigr)s(t) = -\bigl(V_0V_0^T \otimes I\bigr)\tilde A(t)\bigl(V_0 \otimes U(t)\bigr)s(t) + \bigl(V_0V_0^T \otimes I\bigr)\tilde f,
  \end{equation}
  such that $U(t)$ has orthonormal columns for each $t$, with $U(0) = U_0$, $s(0) = s_0$.
\item[--] Determine $s^+_0 = s(h)$ as solution of
\begin{equation}\label{eq:splitting2}
    \ddop{t}(V_0 \otimes U_1)s(t) = (V_0V_0^T \otimes U_1U_1^T)\tilde A(t)(V_0 \otimes U_1)s(t) - (V_0V_0^T\otimes U_1U_1^T)\tilde f, 
  \end{equation}
  with $s(0) =  s^+_1$.
\item[--] Determine $V_1 = V(h)$, $s_1 = s(h)$ as solution of 
 \begin{equation}\label{eq:splitting3}
    \ddop{t}\bigl(V(t) \otimes U_1 \bigr)s(t) = - \bigl(I \otimes U_1U_1^T \bigr)\tilde A(t) \bigl(V(t) \otimes U_1\bigr)s(t) + \bigl(I\otimes U_1U_1^T \bigr)\tilde f
  \end{equation}
  such that $V(t)$ has orthonormal columns for each $t$,
  with $V(0) = V_0$, $s(0) =  s^+_0$.
\end{itemize}
\end{subequations}

\medskip

Altogether, this yields $u_1 = (V_1\otimes U_1) s_1$. Note that the orthogonality requirements on $U$ and $V$ in the first and third step can be enforced by solving for $K(t) =(I \otimes U(t)) s(t)$ and $L(t) =( V(t) \otimes I) s(t)$, respectively, and then factorizing the results (e.g., using QR decomposition) at time $h$. While \eqref{eq:splitting1} and \eqref{eq:splitting3} are parabolic problems projected to a subspace, \eqref{eq:splitting2} is a \emph{backwards} parabolic problem (projected to a finite-dimensional space) that can in principle be arbitrarily ill-conditioned. As we show next, a suitable combination of time discretizations of the three steps in \eqref{eq:splitting} can mitigate this issue. 

Solving \eqref{eq:splitting1} and \eqref{eq:splitting3} by the backward Euler method and \eqref{eq:splitting2} by the forward Euler method, we obtain the following numerical scheme approximating \eqref{eq:splitting}:

\medskip

\begin{subequations}\label{eq:euler}
\begin{itemize}[leftmargin=1em,itemsep=1em]
\item[--] Backwards Euler step for \eqref{eq:splitting1}: Solve
\begin{equation}\label{eq:euler1}
  (V_0 \otimes I ) k_1 + h( V_0 V_0^T \otimes I)\tilde A_1 (V_0 \otimes I) k_1  = (V_0 \otimes U_0)s_0 + h(V_0 V_0^T\otimes I)\tilde f_1,
\end{equation}
for $k_1$, factorize $k_1 = (I \otimes U_1) s^+_1$ such that $U_1$ has orthonormal columns.

\item[--] Forward Euler step for the backwards problem \eqref{eq:splitting2}: update
\begin{equation}\label{eq:euler2}
  s^+_0=  s^+_1 + h(V_0^T \otimes U_1^T)\tilde A_1 (V_0 \otimes U_1) s^+_1 - h(V_0^T\otimes U_1^T)\tilde f_1.
\end{equation}

\item[--] Backwards Euler step for \eqref{eq:splitting3}: Solve
\begin{equation}\label{eq:euler3}
  (I \otimes U_1 )\ell_1 + h(I \otimes U_1 U_1^T)\tilde A_1 (I \otimes  U_1 U_1)\ell_1 = (V_0 \otimes  U_1 ) s^+_0 + h(I\otimes U_1 U_1^T)\tilde f_1
\end{equation}
for $\ell_1$, factorize $\ell_1 = (V_1 \otimes I) s_1$ such that $V_1$ has orthonormal columns.

\end{itemize}
\end{subequations}

\medskip

Multiplying \eqref{eq:euler1} by $(I \otimes U_1^T)$ and substituting into \eqref{eq:euler2}, we find that \eqref{eq:euler2} can be rewritten as
  \begin{equation}\label{eq:euler-proj2}
     s^+_0 = (I \otimes U_1^T U_0) s_0.
  \end{equation}
Thus with this combination of discretizations, the backwards step \eqref{eq:splitting2} amounts to a projection in \eqref{eq:euler2}. Let us mention that a different method for avoiding a backward time step, but with similar properties as~\eqref{eq: projector dynamical low rank}, has recently been proposed in~\cite{Ceruti20}.

As we now show, the time discretization \eqref{eq:euler} is in fact closely related to the so called alternating least squares (ALS) low-rank minimization method applied to the discretized version of our variational time stepping scheme \eqref{eq:euler-opt} with starting value $u_0$ at time $t_0$, that is, to the problem
\begin{equation}\label{eq:euler-opt-discr}
 u_1 =  \argmin_{\operatorname{rank}(y) = r} \left\{  \frac{1}{2(t_{1} - t_0)}\|y-u_0\|_{\ell_2(\mathcal{I})}^2 + \frac{1}{2} \langle \tilde A_1 y, y \rangle  - \langle \tilde f_{1}, y\rangle \right\}  
  \end{equation}
where $\tilde A_i = \tilde A(t_i)$ and $\tilde f_i = \tilde f(t_i)$.
 The ALS method for this minimization problem consists in the following iteration:
 
\medskip

\begin{subequations}\label{eq:ALSsweep}
\noindent Given $y_0 = (V_0\otimes U_0)s_0$, repeat for $j = 0,1,2,\dots$:
\begin{itemize}[leftmargin=1em]
\item[--] Solve
  \begin{equation}\label{eq:ALS1}
    (I + h(V_j^T \otimes I) \tilde A_{j+1} (V_i \otimes I)) k_{j+1} = (I \otimes U_j)s_j + h(V_j^T\otimes I)\tilde f_{j+1}
  \end{equation}
  for $k_{j+1}$, then factorize $k_{j+1} = (I \otimes U_{j+1}){s}^+_{j+1}$.
  
  \smallskip
\item Solve
  \begin{equation}\label{eq:ALS2}
    (I + h(I \otimes U_{j+1}^T) \tilde A_{j+1} (I \otimes U_{j+1})) \ell_{j+1} = (V_j \otimes U_{j+1}^T U_j)s_j+ h(I\otimes U_{j+1}^T)\tilde f_{j+1}
  \end{equation}
  for $\ell_{j+1}$, factorize $\ell_{j+1} =  (V_{j+1}\otimes I)s_{j+1}$,
  and obtain $y_{j+1} = (V_{j+1} \otimes U_{j+1}) s_{j+1}$.
\end{itemize}
\end{subequations}

\medskip

\begin{proposition}
Let starting values $(V_0\otimes U_0)s_0$ be given. Then for the result $u_1 = (V_1 \otimes U_1)s_1$ of \eqref{eq:euler} and the result $y_1$ of \eqref{eq:ALSsweep} after a single step, we have $u_1 = y_1$.
\end{proposition}

\begin{proof}
Factoring out $(V_0 \otimes I)$ in~\eqref{eq:euler1} immediately yields~\eqref{eq:ALS1}. 
Similarly factoring out $(I \otimes U_1)$ in ~\eqref{eq:euler3} and using the projection form~\eqref{eq:euler-proj2} of the second step gives~\eqref{eq:ALS2}.
\end{proof}

\section{Outlook}\label{sec:outlook}

We expect that the obtained existence and uniqueness result is applicable to dynamical low-rank tensor approximations~\cite{Koch10,Lubich13,Lubich15} of higher-dimensional parabolic problems in suitable low-rank formats. 
Beyond the intrinsic interest of parabolic evolution equations under low-rank constraints (or on more general manifolds), the dynamical low-rank approach can also be of interest as an algorithmic component in approximation schemes involving rank adaptivity. 
For instance, with constant right-hand side $f$, performing the low-rank evolution for $u' + Au = f$ to sufficiently large times yields an approximation of $A^{-1} f$. The approach considered here can thus be used in the construction of preconditioners for low-rank approximations of elliptic problems with strongly anisotropic diffusion, where existing methods for Laplacian-type operators are less efficient. A related further question is under what conditions the evolution on $\Mc_r$ approaches the one on the full space $V$ as $r\to \infty$. While results of this type are available for finite-dimensional problems \cite{Feppon18} and for Schr\"odinger-type evolution equations \cite{Conte10}, in our present setting this issue remains open.

\appendix

\section{The fixed-rank manifold in Hilbert space}

\subsection{Local manifold structure}\label{app: local manifold structure}

The goal of this section is to characterize the set $\Mr$ defined in~\eqref{eq: low-rank manifold} locally as an embedded submanifold of $L_2(\Omega)$ by using the submersion theorem. In order to generalize the known arguments from the finite matrix case~\cite[Example~8.14]{Lee2003} it is convenient to consider instead of $\Mr$ the set
\[
\Mc = \{X \in \ell_2(\N^2) : \rank{X} = r\}
\]
of fixed rank-$r$ \emph{infinite} matrices in the real tensor product Hilbert space $\ell_2(\N^2) = {\ell_2(\N) \otimes \ell_2(\N)}$, endowed with the inner product
\[
\langle X, Y \rangle_{\ell_2(\N^2)} = \sum_{i,j} X_{ij} Y_{ij}.
\]
The space $\ell_2(\N^2)$ is isometrically isomorphic to $L_2(\Omega)$, $\Omega = (0,1)^2$, by means of sequence representations with respect to a fixed tensor product orthonormal basis in $L_2(\Omega)$, and $\mathcal M \subset \ell_2(\N^2)$ corresponds precisely to $\Mr \subset L_2(\Omega)$.

As with finite matrices, we can identify the elements of $\ell_2(\N^2)$ as compact (Hilbert-Schmidt) linear operators on $\ell_2(\N)$. For a fixed $X \in \Mc$ we then have the singular value decomposition
\begin{equation}\label{eq: SVD in ell_2}
  X = \sum_{k=1}^r \sigma_k^{} u_k^1 \otimes u_k^2, \quad \sigma_1 \ge \sigma_2 \ge \dots \ge \sigma_r > 0,\ (u_k^1,u_\ell^1)_{\ell_2}= (u_k^2, u_\ell^2)_{\ell_2} = \delta_{k,\ell}.
\end{equation}
Let $\mathcal U_1 = \spn\{ u_1^1,\dots,u_r^1 \}$ and $\mathcal U_2 = \spn\{ u_1^2,\dots,u_r^2 \}$ denote the $r$-dimensional column and row space of $X$, respectively, and $P_1, P_2$ the corresponding orthogonal projections. We can decompose any $Y \in \ell_2(\N^2)$ into the four mutually orthogonal parts
\[
A = P_1 Y P_2, \quad B = P_1 Y (I - P_2), \quad C = (I - P_1)YP_2, \quad D = (I - P_1)
Y(I- P_2).
\]
For any $Y$ in the open ball $\mathcal O \coloneqq \{ Y \in \ell_2(\N^2) \colon \| X - Y \|_{\ell_2(\N^2)} < \sigma_r \}$ the component $A = P_1 Y P_2$ defines an invertible operator from $\mathcal U_2$ to $\mathcal U_1$, since $P_1 X P_2$ obviously is such an invertible operator and its distance to the singular operators is precisely $\sigma_r$.
Therefore, we can consider the map
\[
g : \mathcal O \to \mathcal U_1^\perp \otimes \mathcal U_2^\perp, \quad g(Y)= D - C A^{-1} B,
\]
for which we have the following result.

\begin{proposition}\label{prop: local submersion}
The map $g$ is a submersion and $g^{-1}(0) = \Mc \cap \mathcal O$.
\end{proposition}

\begin{proof}
The derivative of $g$ at $Y = A + B + C + D \in \mathcal O$ is the linear map
\begin{gather}\label{eq: g'(Y)}
g'(Y) : \ell_2(\N^2) \to \mathcal U_1^\perp \otimes \mathcal U_2^\perp, \\
\Delta_A + \Delta_B + \Delta_C + \Delta_D \mapsto \Delta_D - \Delta_C A^{-1}B + CA^{-1}\Delta_A A^{-1} B - CA^{-1} \Delta_B,
 \end{gather}
 where $\Delta_A, \Delta_B, \Delta_C, \Delta_D$ correspond to perturbations in the blocks $A,B,C,D$, respectively. We have used that $(A + \Delta_A)^{-1} = A^{-1} - A^{-1} \Delta_A A^{-1} + O(\| \Delta_A \|^2_{\ell_2(\N^2)})$. Obviously, $g'(Y)$ is surjective for any $Y \in \mathcal O$ and depends continuously on $Y$. Hence $g$ is a submersion.
 
To show that $g^{-1}(0) = \Mc \cap \mathcal O$ we note that any $z \in \ell_2(\N)$ can be decomposed into $z = v+w$ with $v\in \mathcal U_2$ and $w \in \mathcal U_2^\perp$. For $Y$ as above we then have
\[
Yz = Y(v+w)=Av + Bw + Cv + Dw.
\]
Here $Av + Bw \in \mathcal U_1$ and $Cv + Dw \in \mathcal U_1^\perp$ are orthogonal to each other. This implies that $\mathcal U_2$ intersects the null space of $Y$ only trivially (confirming $\rank(Y) \ge k$). Also it shows that $Y z = 0$ if and only if $v = -A^{-1}B w$ and $(D-CA^{-1}B)w=0$. In particular, if $D - CA^{-1}B = 0$, the null space of $Y$ has co-dimension at most $k$, which yields $\rank(Y) = k$. If, on the other hand, $D-CA^{-1}B \neq 0$ we can find $z$ with $Yz \neq 0$ and $w \neq 0$. Since this $z$ is linearly independent from $\mathcal U_2$, one concludes that $\rank (Y) > k$ in this case. 
\end{proof}

It is a standard consequence of Proposition~\ref{prop: local submersion} that $g^{-1}(0) = \Mc \cap \mathcal O$ is a submanifold of the Hilbert space $\ell_2(\N^2)$, and indeed an infinitely smooth one (since $g$ is infinitely smooth); see~\cite[Thm.~73.C]{Zeidler88}. The tangent space at $X$ is the null space of $g'(X)$. Since $g'(X)$ is just the orthogonal projector $Z \mapsto (I-P_1)Z (I - P_2)$ onto $\mathcal U_1^\perp \otimes \mathcal U_2^\perp$ ($B$ and $C$ in~\eqref{eq: g'(Y)} are zero at $Y=X$), this gives
\[
T_X \Mc = (\mathcal U_1^\perp \otimes \mathcal U_2^\perp)^\perp = \mathcal U_1 \otimes \ell_2(\N) + \ell_2(\N) \otimes \mathcal U_2,
\]
which matches the definition~\eqref{eq: tangent space} of $T_u \Mr \subseteq L_2(\Omega)$. We note that by the generalized inverse function theorem~\cite[Thm.~43.C]{Zeidler85}, there exists a $C^1$-homeomorphism 
from a neighborhood of zero in $T_X \Mc$ to a neighborhood (in the subspace topology) of $X$ in $\Mc$, which is also an immersion, and therefore provides a local embedding of $\Mc$ in $\ell_2(\N^2)$.

\subsection{Curvature and projection bounds in Hilbert space}

We now generalize known curvature bounds for finite-dimensional fixed-rank matrix manifolds to the Hilbert space case. As in the previous subsection we first consider $\Mc = \{X \in \ell_2(\N^2) : \rank(X) = r\}$. 
Interpreting the infinite matrices $X \in \ell_2(\N^2)$ as linear operators on $\ell_2(\N)$, it will be important to also consider their spectral norm
\[
  \|X\|_{\ell_2(\N)\to\ell_2(\N)} = \sup_{\|w\|_{\ell_2(\N)} \le 1} \|Xw\|_{\ell_2(\N)}
  \]
  (it equals $\sigma_1$ in the SVD~\eqref{eq: SVD in ell_2}). Then the inequality
\[
  \|XY\|_{\ell_2(\N^2)} \le \|X\|_{\ell_2(\N)\to\ell_2(\N)} \|Y\|_{\ell_2(\N^2)}
\]
holds for all $X,Y\in \ell_2(\N^2)$.

The proof of the following bounds is adapted from the analogous result for finite matrices~\cite[Lemmas~4.1\,\&\,4.2]{Wei16a}.

\begin{lemma}
  \label{thm:curvature}
   Let $X, \hat X \in \Mc \subset \ell_2(\N^2)$, where $X$ has smallest nonzero singular value $\sigma_r (X) >0$. 
   Then the tangent space projections satisfy the Lipschitz-like bound
  \beq
    \label{eq:curvature}
    \|P_{\hat X}(Z) - P_X(Z)\|_{\ell_2(\N^2)} \le \frac{2 }{\sigma_r(X)} \|\hat X-X\|_{\ell_2(\N)\to\ell_2(\N)} \|Z\|_{\ell_2(\N^2)}
  \eeq
  for all $Z \in \ell_2(\N^2)$,
  and 
    \beq
    \label{eq:curvature2}
    \|(I-P_{\hat X})(X-\hat X)\|_{\ell_2(\N^2)} \le \frac{1}{\sigma_r(X)} \|\hat X-X\|^2_{\ell_2(\N^2)}.
  \eeq
\end{lemma}

\newcommand{\QU}{Q_1}
\newcommand{\QhU}{\hat Q_1}
\newcommand{\QV}{Q_2}
\newcommand{\QhV}{\hat Q_2}
\newcommand{\QUC}{Q_1^\bot}
\newcommand{\QhUC}{\hat Q_1^\bot}
\newcommand{\QVC}{\hat Q_2^\bot}
\newcommand{\QhVC}{\hat Q_2^\bot}

\newcommand{\nsp}{{\ell_2(\N)\to\ell_2(\N)}}
\newcommand{\nf}{{\ell_2(\N^2)}}

\begin{proof}
 We have the singular value decompositions $X = U_1^{} S U_2^*$ and $\hat X = \hat U_1^{} \hat S \hat U_2^*$ with infinite matrices $U_i, \hat U_i\in \ell_2(\N)\otimes \R^r$, $i = 1,2$, each of which has orthonormal columns and thus represents a partial isometry, and invertible diagonal matrices $S, \hat S \in \R^{r\times r}$. In this proof, we write $\ell_2^r$ for $\R^r$ with the $\ell_2$-norm and abbreviate $\norm{\cdot} = \norm{\cdot}_\nsp$. For $i=1,2$, we define orthogonal projections on $\ell_2(\N)$ by $Q_i = U_i^{} U_i^*$, $Q_i^\bot = I - Q_i$, as well as $\hat Q_i = \hat U_i^{} \hat U_i^*$, $\hat Q_i^\bot = I - \hat Q_i$. We first show that
 \begin{equation}\label{projdiff}
   \max\bigl\{ \norm{ \QU - \QhU},  \norm{ \QV - \QhV}  \bigr\}  \leq \frac{\norm{X - \hat X}}{\sigma_r(X)} .
 \end{equation}
 It suffices to consider $\QU - \QhU$, with the same estimate for $\QV-\QhV$ following analogously.
 Note first that  $\QU - \QhU = (\QhU + \QhUC) (\QU - \QhU)( \QU + \QUC ) = \QhUC \QU^{} - \QhU^{} \QUC$, which by orthogonality implies
 \[
    \norm{  \QU - \QhU } = \max \bigl \{ \norm{\QhUC \QU^{}} , \norm{\QhU^{} \QUC} \bigr\}.
 \]
 By a similar argument as in the matrix case~\cite[Thm.~2.5.1]{Golub13}, we now show that both norms in the above maximum are in fact equal. For any $x \in \R^r$ with $\norm{x}_{\ell_2^r} = 1$,
 \[
   \norm{\QhU U_1 x}_{\ell_2(\N)}^2 + \norm{\QhUC U_1^{} x}_{\ell_2(\N)}^2 = 1 =  \norm{\QU \hat U_1 x}_{\ell_2(\N)}^2 + \norm{\QUC \hat U_1^{} x}_{\ell_2(\N)}^2 ,
 \]
 and thus
 \[
 \begin{aligned}
   \norm{ \QhUC \QU^{} }^2 &= \max_{\norm{x}_{\ell_2^r}=1} \norm{ \QhUC U_1^{} x}_{\ell_2(\N)}^2 = 1 -   \min_{\norm{x}_{\ell_2^r}=1} \norm{ \QhU U_1^{} x}_{\ell_2(\N)}^2  \\
     & = 1 - \sigma_r(\hat U_1^* U_1^{})^2 = 1 -   \min_{\norm{x}_{\ell_2^r}=1} \norm{ \QU \hat U_1 x}_{\ell_2(\N)}^2   = \norm{ \QUC \QhU^{}}^2 = \norm{  \QhU^{} \QUC }^2.
  \end{aligned}
 \]
 Using that $U_1^{} U_1^* = X U_2^{} S^{-1} U_1^*$ since $U_2^* U_2^{} = I$, we therefore obtain
 \[
 \begin{aligned}
   \norm{ \QU - \QhU } &= \norm{ \QhUC \QU^{}} = \norm{ \QhUC X U_2^{} S^{-1} U_1^*} 
     = \norm{ \QhUC ( \hat X - X) U_2^{} S^{-1} U_1^*}  \\
     & \leq \norm{ \QhUC} \norm{ \hat X - X } \norm{U_2}_{\ell_2^r \to \ell_2(\N)} \norm{S^{-1}}_{\ell_2^r\to\ell_2^r} \norm{ U_1^* }_{\ell_2(\N)\to \ell_2^r} \\
       &= \frac{ \norm{ \hat X - X }}{\sigma_r(X)},
   \end{aligned}
 \]
 and hence~\eqref{projdiff}.
 
 To show~\eqref{eq:curvature}, we observe that
 \[
 \begin{aligned}
   P_{\hat X}(Z) - P_X (Z) &= \QhU Z + Z \QhV - \QhU Z \QhV - \QU Z - Z \QV + \QU Z \QV  \\
    &= ( \QhU^{} - \QU^{} ) Z \QVC + \QhUC Z (\QhV^{} - \QV^{}),
    \end{aligned}
 \]
 and with~\eqref{projdiff},
 \[
 \begin{aligned}
   \norm{  (P_{\hat X} - P_X^{}) (Z) }_\nf &\leq \norm{\QhU - \QU } \norm{Z}_\nf \norm{\QVC} + \norm{\QhUC} \norm{Z}_\nf \norm{\QhV - \QV} \\
    &\leq \frac{2 \norm{ \hat X - X} }{\sigma_r(X)} \norm{Z}_\nf.
    \end{aligned}
 \]
 For~\eqref{eq:curvature2}, we similarly rewrite
 \[
 \begin{aligned}
   (I - P_{\hat X} ) (X) &= (P_X^{} - P_{\hat X})(X) = ( \QU^{} - \QhU^{}) X \QhVC + \QUC X ( \QV^{} - \QhV^{}) \\
     & = ( \QU^{} - \QhU^{}) X \QhVC = ( \QU^{} - \QhU^{}) (\hat X - X )\QhVC,
  \end{aligned}
 \]
 where we have used $\QUC X = 0$ and $\hat X \QhVC = 0$. This gives
 \[
   \norm{ (I - P_{\hat X} ) (X)  }_\nf \leq \norm{ \QU-\QhU } \norm{ \hat X - X}_\nf \norm{\QhV} \leq \frac{\norm{\hat X - X}}{\sigma_r(X)} \norm{\hat X - X}_\nf
 \]
 and thus \eqref{eq:curvature2}.
\end{proof}

Applying Lemma \ref{thm:curvature} to sequence representations with respect to a tensor product orthonormal basis of $L_2(\Omega)$, $\Omega = (0,1)^2$, immediately gives the following result that was used in the main part of the paper.

\begin{corollary}\label{cor: curvature in H norm}
Let $u,v \in \Mc_r \subset H =  L_2(0,1)\otimes L_2(0,1)$ and $w \in H$, and let $\rho$ be a lower bound on the smallest singular value of $u$ in $H$. Then
\[
   \|(P_u - P_v)w\|_{H} \le \frac{\|u - v\|_{H} }{\rho}\| w \|_{H}
\]
and
\[
   \|(I - P_v)(u-v)\|_{H} \le \frac{\|u - v\|^2_{H}}{\rho}.
\]
\end{corollary}

We further provide a proof for the estimate~\eqref{eq: H1 projection bound}. 
\begin{proposition}\label{cor: bound on projection in V norm}
For $u \in \Mc_r \cap V$ with $V = H^1_0(\Omega)$, $\Omega = (0,1)^2$, 
we have
\begin{equation*}
\| P_u \|_{V\to V} \le \left(1 + \frac{r}{\sigma_r(u)^2} \| u \|_V \right)^{\frac{1}{2}}.
\end{equation*}
\end{proposition}
\begin{proof}
Let $\phi \in H_0^1(0,1)$. For the $L_2$-orthogonal projection $P_1$ on the span of the left singular vectors $u_1^1,\dots,u^1_r$, the estimate~\eqref{eq:bound on factor H1 norm} yields
\begin{equation}\label{eq: bound on projection}
\begin{aligned}
\| (\partial_1 \circ P_1)\phi) \|_{L_2(0,1)}^2 &= \left\| \sum_{k=1}^r \langle u_k^1, \phi \rangle_{L_2(0,1)} \partial_1 u_k^1 \right\|^2_{L_2(0,1)} \\ &\le \left( \sum_{k=1}^r \abs{\langle \phi, u_k^1 \rangle}^2 \right) \left( \sum_{k=1}^r \| \partial_1 u_k^1 \|_{L_2(0,1)}^2 \right) \\ &\le r \frac{1}{\sigma_r(u)^2} \| u \|_V \| \phi \|_{L_2(0,1)}^2.
\end{aligned}
\end{equation}
Since $\| \phi \|_{L_2(0,1)} \le \| \phi \|_{H_0^1(0,1)}$ by the Poincar\'e inequality, this shows
\[
\| P_1 \|_{H_0^1(0,1) \to H_0^1(0,1)} \le \frac{\sqrt{r}}{\sigma_r(u)} \| u \|_{V}.
\]
Using~\eqref{eq: tangent space projector} we can write
\(
P_u = I \otimes P_2 + P_1 \otimes (I - P_2),
\)
which due to~\eqref{eq: bound on projection} gives
\begin{align*}
\| P_u v \|_{H_0^1(0,1) \otimes L_2(0,1)}^2 &= \| (I \otimes P_2)v \|_{H_0^1(0,1) \otimes L_2(0,1)}^2 +  \| (P_1 \otimes (I-P_2))v \|_{H_0^1(0,1) \otimes L_2(0,1)}^2 \\
&\le \| v \|_{H_0^1(0,1) \otimes L_2(0,1)}^2 + \frac{r}{\sigma_r(u)^2} \| u \|_V^2 \| v \|_{H_0^1(0,1) \otimes L_2(0,1)}^2
\end{align*}
for any $v \in V$, where we use that the operator norm of a tensor product operator equals the product of operator norms; see,~e.g.,~\cite[Prop.~4.150]{Hackbusch19}. The norm $\| P_u v \|_{L_2(0,1) \otimes H_0^1(0,1)}^2$ can be estimated in the same way, so in summary we have
\begin{equation*}
\| P_u v \|_{V}^2 \le \left(1 + \frac{r}{\sigma_r(u)^2} \| u \|_V^2 \right) \| v \|_V^2,
\end{equation*}
as asserted.
\end{proof}
In the proof we have used the Poincar\'e inequality $\| \phi \|_{L_2(0,1)} \le \| \phi \|_{H_0^1(0,1)}$ on the interval $(0,1)$. When a general domain $\Omega = (a_1,b_1) \times (a_2,b_2)$ is considered, one can obtain a similar estimate $\| P_u \|_{V\to V}^2 \le 1 + \bar {c}^2 \frac{r}{\sigma_r(u)^2} \| u \|_V^2$ where $\bar c$ is the maximum of the Poincar\'e constants of both intervals.

\bibliographystyle{plain}
\bibliography{BEKUparabolic}

\begin{thebibliography}{10}

\bibitem{Arnold14}
A.~Arnold and T.~Jahnke.
\newblock On the approximation of high-dimensional differential equations in
  the hierarchical {T}ucker format.
\newblock {\em BIT}, 54(2):305--341, 2014.

\bibitem{Bachmayr17}
M.~Bachmayr and A.~Cohen.
\newblock Kolmogorov widths and low-rank approximations of parametric elliptic
  {PDE}s.
\newblock {\em Math. Comp.}, 86(304):701--724, 2017.

\bibitem{Bachmayr15}
M.~Bachmayr and W.~Dahmen.
\newblock Adaptive near-optimal rank tensor approximation for high-dimensional
  operator equations.
\newblock {\em Found. Comput. Math.}, 15(4):839--898, 2015.

\bibitem{Bachmayr16b}
M.~Bachmayr and W.~Dahmen.
\newblock Adaptive low-rank methods: problems on {S}obolev spaces.
\newblock {\em SIAM J. Numer. Anal.}, 54(2):744--796, 2016.

\bibitem{Bachmayr16}
M.~Bachmayr, R.~Schneider, and A.~Uschmajew.
\newblock Tensor networks and hierarchical tensors for the solution of
  high-dimensional partial differential equations.
\newblock {\em Found. Comput. Math.}, 16(6):1423--1472, 2016.

\bibitem{Bardos10}
C.~Bardos, I.~Catto, N.~Mauser, and S.~Trabelsi.
\newblock Setting and analysis of the multi-configuration time-dependent
  {H}artree-{F}ock equations.
\newblock {\em Arch. Ration. Mech. Anal.}, 198(1):273--330, 2010.

\bibitem{Bardos09}
C.~Bardos, I.~Catto, N.~J. Mauser, and S.~Trabelsi.
\newblock Global-in-time existence of solutions to the multiconfiguration
  time-dependent {H}artree-{F}ock equations: a sufficient condition.
\newblock {\em Appl. Math. Lett.}, 22(2):147--152, 2009.

\bibitem{Bartels15}
S.~Bartels.
\newblock {\em Numerical methods for nonlinear partial differential equations}.
\newblock Springer, Cham, 2015.

\bibitem{Boiveau19}
T.~Boiveau, V.~Ehrlacher, A.~Ern, and A.~Nouy.
\newblock Low-rank approximation of linear parabolic equations by space-time
  tensor {G}alerkin methods.
\newblock {\em ESAIM: M2AN}, 53(2):635--658, 2019.

\bibitem{Conte10}
D.~Conte and C.~Lubich.
\newblock An error analysis of the multi-configuration time-dependent {H}artree
  method of quantum dynamics.
\newblock {\em M2AN Math. Model. Numer. Anal.}, 44(4):759--780, 2010.

\bibitem{Dahmen16}
W.~Dahmen, R.~DeVore, L.~Grasedyck, and E.~S\"{u}li.
\newblock Tensor-sparsity of solutions to high-dimensional elliptic partial
  differential equations.
\newblock {\em Found. Comput. Math.}, 16(4):813--874, 2016.

\bibitem{DeLathauwer00}
L.~De~Lathauwer, B.~De~Moor, and J.~Vandewalle.
\newblock A multilinear singular value decomposition.
\newblock {\em SIAM J. Matrix Anal. Appl.}, 21(4):1253--1278, 2000.

\bibitem{Dirac30}
P.~A.~M. Dirac.
\newblock Note on exchange phenomena in the {T}homas atom.
\newblock {\em Math. Proc. Cambridge Philos. Soc.}, 26(3):376--385, 1930.

\bibitem{Einkemmer18}
L.~Einkemmer and C.~Lubich.
\newblock A low-rank projector-splitting integrator for the {V}lasov-{P}oisson
  equation.
\newblock {\em SIAM J. Sci. Comput.}, 40(5):B1330--B1360, 2018.

\bibitem{Emmrich04}
E.~Emmrich.
\newblock {\em Gew{\"o}hnliche und Operator-Differentialgleichungen}.
\newblock Vieweg, Wiesbaden, 2004.

\bibitem{Falco19}
A.~Falc\'{o}, W.~Hackbusch, and A.~Nouy.
\newblock On the {D}irac-{F}renkel variational principle on tensor {B}anach
  spaces.
\newblock {\em Found. Comput. Math.}, 19(1):159--204, 2019.

\bibitem{Feppon18}
F.~Feppon and P.~F.~J. Lermusiaux.
\newblock A geometric approach to dynamical model order reduction.
\newblock {\em SIAM J. Matrix Anal. Appl.}, 39(1):510--538, 2018.

\bibitem{Ceruti20}
Ceruti G. and Lubich C.
\newblock An unconventional robust integrator for dynamical low-rank
  approximation.
\newblock arXiv preprint arXiv:2010.02022, 2020.

\bibitem{Girault79}
V.~Girault and P.-A. Raviart.
\newblock {\em Finite element approximation of the {N}avier-{S}tokes
  equations}.
\newblock Springer-Verlag, Berlin-New York, 1979.

\bibitem{Golub13}
G.~H. Golub and C.~F. Van~Loan.
\newblock {\em Matrix computations}.
\newblock Johns Hopkins University Press, Baltimore, MD, fourth edition, 2013.

\bibitem{Grasedyck04}
L.~Grasedyck.
\newblock Existence and computation of low {K}ronecker-rank approximations for
  large linear systems of tensor product structure.
\newblock {\em Computing}, 72(3-4):247--265, 2004.

\bibitem{Hackbusch19}
W.~Hackbusch.
\newblock {\em Tensor spaces and numerical tensor calculus}.
\newblock Springer, Cham, second edition, 2019.

\bibitem{Hackbusch09}
W.~Hackbusch and S.~K\"{u}hn.
\newblock A new scheme for the tensor representation.
\newblock {\em J. Fourier Anal. Appl.}, 15(5):706--722, 2009.

\bibitem{Haegeman11}
J.~Haegeman, J.~I. Cirac, T.~J. Osborne, I.~Pi{\v{z}}orn, H.~Verschelde, and
  F.~Verstraete.
\newblock Time-dependent variational principle for quantum lattices.
\newblock {\em Phys. Rev. Lett.}, 107(7):070601, 2011.

\bibitem{Kazashi20}
Y.~Kazashi and F.~Nobile.
\newblock Existence of dynamical low rank approximations for random semi-linear
  evolutionary equations on the maximal interval.
\newblock arXiv preprint arXiv:2002.02356, 2020.

\bibitem{Kieri16}
E.~Kieri, C.~Lubich, and H.~Walach.
\newblock Discretized dynamical low-rank approximation in the presence of small
  singular values.
\newblock {\em SIAM J. Numer. Anal.}, 54(2):1020--1038, 2016.

\bibitem{Koch07}
O.~Koch and C.~Lubich.
\newblock Dynamical low-rank approximation.
\newblock {\em SIAM J. Matrix Anal. Appl.}, 29(2):434--454, 2007.

\bibitem{Koch07b}
O.~Koch and C.~Lubich.
\newblock Regularity of the multi-configuration time-dependent {H}artree
  approximation in quantum molecular dynamics.
\newblock {\em M2AN Math. Model. Numer. Anal.}, 41(2):315--331, 2007.

\bibitem{Koch10}
O.~Koch and C.~Lubich.
\newblock Dynamical tensor approximation.
\newblock {\em SIAM J. Matrix Anal. Appl.}, 31(5):2360--2375, 2010.

\bibitem{Koch11}
O.~Koch and C.~Lubich.
\newblock Variational-splitting time integration of the multi-configuration
  time-dependent {H}artree-{F}ock equations in electron dynamics.
\newblock {\em IMA J. Numer. Anal.}, 31(2):379--395, 2011.

\bibitem{Lee2003}
J.~M. Lee.
\newblock {\em Introduction to smooth manifolds}.
\newblock Springer-Verlag, New York, 2003.

\bibitem{Lubich08}
C.~Lubich.
\newblock {\em From quantum to classical molecular dynamics: reduced models and
  numerical analysis}.
\newblock European Mathematical Society (EMS), Z\"{u}rich, 2008.

\bibitem{Lubich14}
C.~Lubich and I.~V. Oseledets.
\newblock A projector-splitting integrator for dynamical low-rank
  approximation.
\newblock {\em BIT}, 54(1):171--188, 2014.

\bibitem{Lubich15}
C.~Lubich, I.~V. Oseledets, and B.~Vandereycken.
\newblock Time integration of tensor trains.
\newblock {\em SIAM J. Numer. Anal.}, 53(2):917--941, 2015.

\bibitem{Lubich13}
C.~Lubich, T.~Rohwedder, R.~Schneider, and B.~Vandereycken.
\newblock Dynamical approximation by hierarchical {T}ucker and tensor-train
  tensors.
\newblock {\em SIAM J. Matrix Anal. Appl.}, 34(2):470--494, 2013.

\bibitem{Mena18}
H.~Mena, A.~Ostermann, L.-M. Pfurtscheller, and C.~Piazzola.
\newblock Numerical low-rank approximation of matrix differential equations.
\newblock {\em J. Comput. Appl. Math.}, 340:602--614, 2018.

\bibitem{Meyer90}
H.-D. Meyer, U.~Manthe, and L.~S. Cederbaum.
\newblock The multi-configurational time-dependent {Hartree} approach.
\newblock {\em Chem. Phys. Lett.}, 165(1):73--78, 1990.

\bibitem{Musharbash18}
E.~Musharbash and F.~Nobile.
\newblock Dual dynamically orthogonal approximation of incompressible {N}avier
  {S}tokes equations with random boundary conditions.
\newblock {\em J. Comput. Phys.}, 354:135--162, 2018.

\bibitem{Musharbash15}
E.~Musharbash, F.~Nobile, and T.~Zhou.
\newblock Error analysis of the dynamically orthogonal approximation of time
  dependent random {PDE}s.
\newblock {\em SIAM J. Sci. Comput.}, 37(2):A776--A810, 2015.

\bibitem{Oseledets11}
I.~V. Oseledets.
\newblock Tensor-train decomposition.
\newblock {\em SIAM J. Sci. Comput.}, 33(5):2295--2317, 2011.

\bibitem{Ostermann19}
A.~Ostermann, C.~Piazzola, and H.~Walach.
\newblock Convergence of a low-rank {L}ie-{T}rotter splitting for stiff matrix
  differential equations.
\newblock {\em SIAM J. Numer. Anal.}, 57(4):1947--1966, 2019.

\bibitem{Sapsis09}
T.~P. Sapsis and P.~F.~J. Lermusiaux.
\newblock Dynamically orthogonal field equations for continuous stochastic
  dynamical systems.
\newblock {\em Phys. D}, 238(23-24):2347--2360, 2009.

\bibitem{Schmidt1907}
E.~Schmidt.
\newblock Zur {T}heorie der linearen und nichtlinearen {I}ntegralgleichungen.
\newblock {\em Math. Ann.}, 63(4):433--476, 1907.

\bibitem{Schneider14}
R.~Schneider and A.~Uschmajew.
\newblock Approximation rates for the hierarchical tensor format in periodic
  {S}obolev spaces.
\newblock {\em J. Complexity}, 30(2):56--71, 2014.

\bibitem{Showalter97}
R.~E. Showalter.
\newblock {\em Monotone operators in {B}anach space and nonlinear partial
  differential equations}.
\newblock American Mathematical Society, Providence, RI, 1997.

\bibitem{Uschmajew20}
A.~Uschmajew and B.~Vandereycken.
\newblock Geometric methods on low-rank matrix and tensor manifolds.
\newblock In P.~Grohs, M.~Holler, and A.~Weinmann, editors, {\em Handbook of
  variational methods for nonlinear geometric data}, pages 261--313. Springer,
  Cham, 2020.

\bibitem{Wei16a}
K.~Wei, J.-F. Cai, T.~F. Chan, and S.~Leung.
\newblock Guarantees for {R}iemannian optimization for low rank matrix
  recovery.
\newblock {\em SIAM J. Matrix Anal. Appl.}, 37(3):1198--1222, 2016.

\bibitem{Zeidler85}
E.~Zeidler.
\newblock {\em Nonlinear functional analysis and its applications. {III}.
  Variational methods and optimization.}
\newblock Springer-Verlag, New~York, 1985.

\bibitem{Zeidler88}
E.~Zeidler.
\newblock {\em Nonlinear functional analysis and its applications. {IV}.
  Applications to Mathematical Physics.}
\newblock Springer-Verlag, New~York, 1988.

\bibitem{Zeidler90}
E.~Zeidler.
\newblock {\em Nonlinear functional analysis and its applications. {II}/{A}.
  Linear monotone operators.}
\newblock Springer-Verlag, New~York, 1990.

\bibitem{Zeidler95}
E.~Zeidler.
\newblock {\em Applied functional analysis. Main principles and their
  applications}.
\newblock Springer-Verlag, New~York, 1995.

\end{thebibliography}

\end{document}